\documentclass[12pt,reqno,a4paper]{amsart}

\usepackage{amsmath, amssymb, amsthm}
\usepackage{enumerate, enumitem}
\usepackage{mathrsfs}
\usepackage{eqlist, array, float}
\usepackage{hyperref}
\usepackage{tikz, forest, tikz-qtree}
\usepackage{comment}
\usepackage{xcolor}

\theoremstyle{plain}
\newtheorem{theorem}{Theorem}[section]
\newtheorem{proposition}[theorem]{Proposition}
\newtheorem{lemma}[theorem]{Lemma}
\newtheorem{corol}[theorem]{Corollary}

\theoremstyle{definition}
\newtheorem{example}[theorem]{Example}

\theoremstyle{definition}

\newtheorem{remark}[theorem]{Remark}

\numberwithin{equation}{section}

\newcommand{\NN}{\mathbb{N}}
\newcommand{\ZZ}{\mathbb{Z}}

\newcommand{\mm}{\mathfrak{m}} 
\newcommand{\BB}{\mathcal{B}} 
 
\newcommand{\ord}{\operatorname{ord}}

\newcommand{\Mod}[1]{\ (\mathrm{mod}\ #1)}

\newcommand\TODO[1]{\textcolor{red}{#1}}
\newcommand\anitha[1]{\textcolor{blue}{#1}}

\begin{document}


\title[Branches of Markoff $m$-triples with two $k$-Fibonacci components]%
      {Branches of Markoff $m$-triples with two $k$-Fibonacci components}

\author[D. Alfaya, L. A. Calvo, P. J. Cazorla, J. Rodrigo \and A. Srinivasan]%
       {David Alfaya*  ***, Luis Ángel Calvo** ***, Pedro-José Cazorla*, 
        Javier Rodrigo* \and Anitha Srinivasan**}

\newcommand{\acr}{\newline\indent}

\address{\llap{*\,}Department of Applied Mathematics,\acr
ICAI School of Engineering, Comillas Pontifical University,\acr
C/Alberto Aguilera 25, 28015 Madrid,\acr Spain.}
\email{dalfaya@comillas.edu, jrodrigo@comillas.edu, pjcazorla@comillas.edu}

\address{\llap{**\,}Department of Quantitative Methods,\acr
ICADE, Comillas Pontifical University,\acr
C/Alberto Aguilera 23, 28015 Madrid,\acr Spain.}
\email{lacalvo@comillas.edu, asrinivasan@icade.comillas.edu}

\address{\llap{***\,}Institute for Research in Technology,\acr
Comillas Pontifical University,\acr
C/Santa Cruz de Marcenado 26, 28015 Madrid,\acr Spain.}
\email{dalfaya@comillas.edu, lacalvo@comillas.edu}

\thanks{\textit{Acknowledgements}.
This research was supported by project CIAMOD (Applications of computational methods and artificial intelligence to the study of moduli spaces, project PP2023\_9) funded by Universidad Pontificia Comillas, and by grant PID2022-142024NB-I00 funded by MCIN/AEI/10.13039/501100011033.}


\begin{abstract}
We study infinite paths of Markoff $m$-triples, that is, solutions to the generalised Markoff equation
\[
x^2+y^2+z^2=3xyz+m,
\]
with $m>0$, with at least two $k$-Fibonacci components.
First, we obtain a complete classification of
Markoff $m$-triples whose last two entries are $k$-Fibonacci numbers and that are not roots of any Markoff trees.
We then prove that every such infinite path is contained in a branch, starting at a triple of the form
\[
\left(\frac{F_k(4r)}{3F_k(2r)},\,F_k(\ell+2r),\,F_k(\ell+4r)\right),
\]
where $r$ is an odd integer, $\ell\in\{1,2,\ldots, 2r\}$ and $3\nmid k$.
These branches are distributed among exactly $2r$ distinct trees.
\end{abstract}

\maketitle
\tableofcontents

\newpage

\section{Introduction}

\noindent Consider the function
\[
  \mm(x,y,z) := x^{2}+y^{2}+z^{2}-3xyz.
\]
Given a non-negative integer $m$, we say that an ordered triple $(a,b,c)\in \mathbb{N}^3$
is a \emph{Markoff \(m\)-triple} if
\[
  \mm(a,b,c)=m.
\]

\noindent This equation, called the Markoff $m$-equation, generalises the classical Diophantine equation introduced by A.~A.~Markoff, corresponding to the case \(m=0\)~\cite{M1,M2}.  
In that classical setting, all ordered solutions are connected via a unique infinite tree, generated by the three \emph{Vieta transformations}:
\begin{equation}\label{def:vieta}
\begin{aligned}
\nu_1(a,b,c) &= (b, c, 3bc - a), \\
\nu_2(a,b,c) &= (a, c, 3ac - b), \\
\nu_3(a,b,c) &= \ord(3ab - c, a, b),
\end{aligned}
\end{equation}
 where $\ord$ denotes the operation of rearranging a triple into non-decreasing order. By Lemma 2.1 of \cite{SC}, one has $3ab-c<b$, whereas $3ab-c$ may be either smaller or larger than $a$. 

\noindent If $m>0$, the above Vieta transformations are still valid and give rise to trees of solutions. As shown in~\cite{SC}, when \(m>0\) the number of distinct \(m\)-trees is equal to the number of \emph{minimal Markoff \(m\)-triples}, namely those ordered triples \((a,b,c) \in \mathbb{N}^3\) such that \(\nu_3(a,b,c)\) has a non-positive component. Equivalently, a Markoff $m$-triple $(a,b,c)$ is minimal if
$$c \ge 3ab.$$

An \emph{infinite path} is an infinite sequence of 
Markoff $m$-triples $\{(a_n,b_n,c_n)\}_{n\ge 0}$ such that each triple is obtained
from the previous one by applying one of the Vieta transformations $\nu_1$ or $\nu_2$.
Equivalently, for every $n\ge 1,$ we have
\begin{equation*}
\nu_3(a_{n+1},b_{n+1},c_{n+1})=(a_n,b_n,c_n).
\end{equation*}

\noindent In particular, we denote by $\BB(a_0,b_0,c_0)$ the \emph{branch rooted at $(a_0,b_0,c_0)$},
namely, the infinite path consisting exclusively of
$\nu_2$-steps, i.e.
$$
  \BB(a_0,b_0,c_0)
  := \bigl\{\nu_2^{\,n}(a_0,b_0,c_0)\mid n\in\mathbb Z_{\ge 0}\bigr\}.
$$
\noindent Two classical examples in the case \(m = 0\) are, the branch \(\BB(1,1,2)\), whose components are Fibonacci numbers of odd index, and the branch \(\BB(2,2,12)\), whose components are Pell numbers of odd index.

In this article, we deal with $k$-Fibonacci numbers, defined for any positive integer $k$ by
\begin{equation}\label{df:k_fibonacci}
\left\{
\begin{aligned}
F_k(0) &= 0, \\
F_k(1) &= 1, \\
F_k(n) &= k F_k(n-1) + F_k(n-2), \quad \text{for } n \ge 2.
\end{aligned}
\right.
\end{equation}
\noindent In the notation of \cite{Lehmer}, $F_k(n)$ corresponds to a Lucas-type sequence $U_n$ with parameters $P = k$ and $Q = -1$. There is an extensive literature on branches of $m$-trees whose three elements are $k$-Fibonacci components: the classical Markoff equation ($m=0$) was studied for Fibonacci numbers ($k=1$) by Luca and the fifth author in \cite{LS}, for Pell numbers ($k=2$) by Kafle, Togbe and Srinivasan in \cite{KST}, for $k>2$ by Gómez, Gómez and Luca in \cite{Gom}, and for Lucas sequences in \cite{AL,RSP}. By contrast, the literature on the Markoff equation for $m>0$ is comparatively scarce: the case of Fibonacci triples ($k=1$) was treated in \cite{ACMRS25a}, and that of $k$-Fibonacci triples with $k\ge2$ in \cite{ACMRS25b}. Finally, Luca~\cite{Luca} analysed Markoff triples ($m=0$) with two Fibonacci components $(k=1)$, describing the branch $\BB(1,1,2)$ and its image under $\nu_1$, and proving that only finitely many such triples can occur outside of these instances.

In this paper, we extend Luca's result by describing the infinite paths of Markoff $m$-triples with at least two $k$-Fibonacci components and $m>0$.

Our first result classifies all non-minimal Markoff \(m\)-triples whose
last two entries are \(k\)-Fibonacci numbers.

\begin{theorem}\label{the:nonminimal_general} Let $(a,b,c)$ be a non-minimal Markoff $m$-triple with $m>0$. Then \(b\) and \(c\) are $k$-Fibonacci numbers if and only if
\[
(a,b,c)=\left(\alpha_{k,r},\,F_k(N-r),\,F_k(N+r)\right),
\]
where $\alpha_{k,r}=\frac{F_k(4r)}{3F_k(2r)},$ \(r\) is an odd integer, \(3\nmid k\), and \(N>3r\); moreover, if \(k\in\{1,2\}\), then \(N\) must be odd.
\end{theorem}

The triples appearing in Theorem~\ref{the:nonminimal_general} naturally give rise to branches in the corresponding \(m\)-trees. More precisely, we call
\[
\BB\bigl(\alpha_{k,r},\,F_k(N-r),\,F_k(N+r)\bigr)
\]
a \emph{principal \((2,k)\)-Fibonacci branch}.
Our second result describes how these principal branches are distributed among distinct $m$-trees.

\begin{theorem}
\label{prop:branches}
Fix an odd integer \(r\ge1\) and assume \(3\nmid k\).
Then the family of Markoff triples
\[
\bigl(\alpha_{k,r},\,F_k(\ell),\,F_k(\ell+2r)\bigr), \qquad \ell\in\NN,\ 
\]
with
\[
\begin{cases}
\ell>2r, & \text{if } k>2,\\[3pt]
\ell\geq 2r \text{ and }\ell\ \text{even}, & \text{if } k\in\{1,2\},
\end{cases}
\]
is distributed among exactly \(2r\) principal $(2,k)$-Fibonacci branches, described as follows:
\begin{enumerate}
\item For each even \(\ell_0\in\{2,4,\dots,2r\}\), the triple
\((F_k(\ell_0),\alpha_{k,r},F_k(\ell_0+2r))\) is minimal. The corresponding branch is
\[
\BB\big(\nu_1(F_k(\ell_0),\alpha_{k,r},F_k(\ell_0+2r))\big)
=
\BB\big(\alpha_{k,r},F_k(\ell_0+2r),F_k(\ell_0+4r)\big).
\]

\item For each odd \(\ell_0\in\{1,3,\dots,2r-1\}\) and \(k\ge4\), the triple
\(\nu_3(F_k(\ell_0),\alpha_{k,r},F_k(\ell_0+2r))\) is minimal. The corresponding branch is
\[
\BB\big(\nu_1(F_k(\ell_0),\alpha_{k,r},F_k(\ell_0+2r))\big)
=
\BB\big(\alpha_{k,r},F_k(\ell_0+2r),F_k(\ell_0+4r)\big).
\]
\end{enumerate}
\end{theorem}

Finally, we prove that these principal branches exhaust all
possible infinite paths with at least two \(k\)-Fibonacci components.

\begin{theorem}\label{thm:only_principal}
Let $m>0$ and $k\ge1$. Every infinite path of Markoff $m$-triples
with at least two $k$-Fibonacci components is contained in a principal $(2,k)$-Fibonacci branch.
\end{theorem}

This paper is organised as follows. In Section~\ref{sec:nonminimals} we prove a general uniqueness result for ordered non-minimal Markoff $m$-triples with fixed last two components. Section~\ref{section:1} gathers the basic properties of $k$-Fibonacci and $k$-Lucas numbers used throughout the paper. In Section~\ref{section:2} we introduce and study principal $(2,k)$-Fibonacci branches, determine when triples of the form $(a,F_k(n-r),F_k(n+r))$ define Markoff $m$-triples with positive parameter $m$, and prove Theorem~\ref{the:nonminimal_general}. In Section~\ref{section:5} we describe how these branches are distributed among distinct $m$-trees, thereby proving Theorem~\ref{prop:branches}. Section~\ref{section:6} shows that no other infinite paths with at least two $k$-Fibonacci components can occur, which yields Theorem~\ref{thm:only_principal}. Finally, Section~\ref{section:examples} contains several examples illustrating the main results.

\section{Uniqueness of non-minimal Markoff $m$-triples $(a,b,c)$ with fixed $b,c$}
\label{sec:nonminimals}

In this section, we prove a general uniqueness statement for non-minimal
Markoff $m$-triples when the last two components are fixed.
This auxiliary result will be used in Section~\ref{section:2}.

 Let $(a,b,c)$ be a Markoff $m$-triple that is not minimal, meaning $c < 3ab$. We define \( a_q = \frac{c}{3b} \). Let $a_p$ be the smaller solution in $x$ of the equation $\mm(x,b,c)=0,$
with $b$ and $c$ fixed.

\begin{remark}
    \label{rmk:usefulnotes} Note that $a_q=\frac{c}{3b}$ provides the threshold between minimal and non-minimal
Markoff $m$-triples with fixed last two components.  Indeed, a Markoff $m$-triple
$(x,b,c)$ is minimal if and only if $c\ge 3xb$, which is equivalent to $x\le a_q$. Moreover, it is easy to show that $a_q < c$ and $a_p \le c$.
\end{remark}

\begin{lemma}
\label{lemma:mlessthanp}
The inequality $a_q < a_p$ holds.
\end{lemma}

\begin{proof}
The function \( \mm(x, b, c) \) is a quadratic in \( x \) with its vertex at \( x = \frac{3}{2}bc \). Since \( \frac{3}{2}bc > c \), the vertex lies to the right of \( c \), so the function is strictly decreasing for \( x \leq c \). We note that both $a_q$ and $a_p$ are smaller than or equal to $c$ and, consequently, in order to prove that \( a_q < a_p \), it suffices to show:
\[
\mm(a_q, b, c) > \mm(a_p, b, c) = 0.
\]

\noindent We compute
\[
\mm\!\left( \frac{c}{3b}, b, c \right) = \left( \frac{c}{3b} \right)^2 + b^2 + c^2 - 3 \cdot \frac{c}{3b} \cdot b \cdot c = \left( \frac{c}{3b} \right)^2 + b^2.
\]
This expression is clearly positive, hence
\[
\mm(a_q, b, c) > 0 = \mm(a_p, b, c),
\]
which implies \( a_q < a_p \), as desired.
\end{proof}
\medskip

\begin{proposition}
\label{prop:smallerinterval}
If \( (a, b, c) \) is a non-minimal Markoff $m$-triple, then
\[
a_p - a_q < \frac{b}{c}.
\]

\end{proposition}

\begin{proof}
If \( a_q + \frac{b}{c} \geq c \), then the inequality \( a_p - a_q < \frac{b}{c} \) holds trivially, since \( a_p \leq c \) and \( a_q < c \). Otherwise, by a similar argument to that of Lemma \ref{lemma:mlessthanp}, it suffices to show that
\[
\mm\!\left(a_q + \frac{b}{c},\, b,\, c\right) < \mm(a_p, b, c) = 0.
\]
Indeed, we compute:
\begin{align*}
\mm\!\left(a_q + \frac{b}{c}, b, c\right)
&= \left( \frac{c}{3b} + \frac{b}{c} \right)^2 + b^2 + c^2 - 3bc\left( \frac{c}{3b} + \frac{b}{c} \right) \\
&= \left( \frac{c}{3b} \right)^2 + \frac{b^2}{c^2} + \frac{2}{3} + b^2 + c^2 - 3bc \left( \frac{c}{3b} + \frac{b}{c} \right) \\
&= \left( \frac{c}{3b} \right)^2 + \frac{b^2}{c^2} + \frac{2}{3} - 2b^2.
\end{align*}

\noindent Since \( (a, b, c) \) is non-minimal and ordered, then $a>\frac{c}{3b}$ and $b \le c$, so that
 \[\mm\!\left(a_q+\frac{b}{c}, b, c\right) < a^2+\frac{5}{3}-2b^2 < 0,\]
where the last inequality always holds except when \( a = b = 1 \). For the triple $(1,1,c)$ to be non-minimal, it follows that $c<3ab=3$. We check both cases, $c=1$ and $c=2$:
\begin{itemize}
  \item If $(a,b,c)=(1,1,1)$, then
    $a_q=\dfrac13$, $a_p$ is the left root of $x^{2}-3x+2=(x-1)(x-2)$, hence
    $a_p=1$.  
    Therefore $a_p-a_q=\dfrac23<\dfrac{b}{c}=1$.

  \item If $(a,b,c)=(1,1,2)$, then
    $a_q=\dfrac23$, $a_p$ is the left root of $x^{2}-6x+5=(x-1)(x-5)$, hence
    $a_p=1$.  
    Therefore $a_p-a_q=\dfrac13<\dfrac{b}{c}=\dfrac12$.
\end{itemize}
\end{proof}
\medskip 

\begin{corol}
\label{propr:unicity}
Let \( a_q = \frac{c}{3b} \), and let \( a_p \) be the left root of the parabola \( \mm(x, b, c) = 0 \). Then, for any fixed pair \( (b, c) \) of positive integers, with $b \le c$, the open interval \( (a_q, a_p) \) contains at most one integer \( a^* \). If such an integer exists, then \( (a^*, b, c) \) is the unique ordered non-minimal Markoff \( m \)-triple with $m>0$ and second and third entries equal to \( b \) and \( c \), respectively.
\end{corol}

\begin{proof}
By Lemma~\ref{lemma:mlessthanp}, the interval \( (a_q, a_p) \) is non-empty. Since we are assuming that $b \le c$, Proposition \ref{prop:smallerinterval} yields that $a_p-a_q < 1$ and thus the interval $(a_q, a_p)$ contains at most one integer.

Moreover, since \( \mm(x, b, c) \) is strictly decreasing on this interval, any integer \( a^* \in (a_q, a_p) \) satisfies
\[
\mm(a^*, b, c) > \mm(a_p, b, c) = 0,
\]
and hence defines a non-minimal Markoff \( m \)-triple. 
\end{proof}
\medskip

\section{\texorpdfstring{$k$-Fibonacci and $k$-Lucas numbers}{Preliminaries on k-Fibonacci numbers and k-Lucas numbers}}\label{section:1}

In this section, we include several useful results related to the $k$-Fibonacci and $k$-Lucas numbers. Many of these have been previously proved in the literature, and we will refer the reader to the relevant sources when needed.

Throughout this paper, we denote
$$
\phi_k = \frac{k + \sqrt{k^2 + 4}}{2} \quad \text{and} \quad \overline{\phi}_k = \frac{k - \sqrt{k^2 + 4}}{2},
$$
so that for each $n\ge 0$, the $n$-th $k$-Fibonacci number can be written through Binet's formula as
\begin{equation}\label{eqn:binet}
F_k(n)=\frac{\phi_k^{\,n}-\overline{\phi}_k^{\,n}}{\sqrt{k^2+4}}\, .
\end{equation}
Since $|\overline{\phi}_k|<1$, it follows immediately from \eqref{eqn:binet} that
\begin{equation}\label{eqn:usefulasympt}
\lim_{n\to\infty}\frac{\phi_k^n}{D_kF_k(n)}=1,
\end{equation}
where, for convenience, we write
\[
D_k=\sqrt{k^2+4}.
\]

\noindent The following lemma and corollary are  proved  in  \cite[Lemma 2.1 and Corollary 2.2]{ACMRS25b}.

\begin{lemma}[Generalization of Vajda's Identity for $k$-Fibonacci numbers] \label{lemma:Vajda} For any integers $a, b$, $k$ and $n$, the following identity holds: $$F_k(n+a)F_k(n+b)-F_k(n)F_k(n+a+b)=(-1)^nF_k(a)F_k(b).$$
\end{lemma}
\medskip

\begin{corol}\label{cor:identities}
The following identities hold for any positive integers $a,b,n:$
\begin{gather}
    F_k(a+b)=F_k(a+1)F_k(b)+F_k(a)F_k(b-1), \label{eq:sum}\, \\
    F_k(a)\leq \frac{1}{k} F_k(a+1), \label{eq:basic_bound}\, \\
    F_k(a) F_k(b)\leq F_k(a+b-1),  \label{eq:bound_product}\, \\
    F_k(a+b-1)\le F_k(a)F_k(b)\left(1+\frac{1}{k^2}\right),   \label{eq:bound_product2}\, \\
    \text{(D’Ocagne; $b\ge a$)}\ \,\,(-1)^aF_k(b-a) = F_k(b)F_k(a+1)-F_k(b+1)F_k(a),\label{eq:DOcagne}\,\\
    (\text{Catalan}) \,\,F_k(n)^2=F_k(n+r)F_k(n-r)+(-1)^{n-r}F_k(r)^2 ,\label{eq:Catalan}\,\\
    (\text{Simson}) \,\,F_k(n)^2=F_k(n+1)F_k(n-1)-(-1)^{n} \label{eq:Simson}\,.
\end{gather}
Moreover, equality holds in the following cases:
\begin{enumerate}
    \item The equality in \eqref{eq:basic_bound} is only attained if $a=1$.
    \item The equality in \eqref{eq:bound_product} is only attained if either $a=1$ or $b=1$.
    \item The equality in \eqref{eq:bound_product2} is only attained if $a=b=2$.
\end{enumerate}
\end{corol}
\medskip
\noindent The following lemma provides a lower bound for the product of two $k$-Fibonacci numbers, depending on the value of $k \ge 1$, and was proved in \cite[Lemma 2.2]{ACMRS25a} and \cite[Lemma 2.5]{ACMRS25b}.
\begin{lemma}
\label{lemma:F(c)<=3F(n)F(m)}
Let $1\leq a \leq b\leq c$ be integers. Then
\begin{enumerate}
\item[\textup{(i)}] \textbf{Case $k=1$.} 
\begin{enumerate}
\item[\textup{(a)}] If $a,b\ge 2$, then $F_{1}(c)\le 3\,F_{1}(a)F_{1}(b) \ \text{if and only if}\ c\le a+b, \\
F_{1}(c)> 3\,F_{1}(a)F_{1}(b) \ \text{if and only if}\ c\ge a+b+1,$
and equality holds if and only if $(a,b,c)=(2,2,4)$.
\item[\textup{(b)}]  $
F_{1}(c)\le 3\,F_{1}(1)F_{1}(1) \ \text{if and only if}\ c\le 4, \\
F_{1}(c)> 3\,F_{1}(1)F_{1}(1) \ \text{if and only if}\ c\ge 5,
$
and equality holds if and only if $c=4$.
\item[\textup{(c)}] If $a=1$ with $b\ge 2$, then
$F_{1}(c)\le 3\,F_{1}(1)F_{1}(b) \ \text{if and only if}\ c\le a+b+1; \,\,
F_{1}(c)> 3\,F_{1}(1)F_{1}(b) \ \text{if and only if}\ c\ge a+b+2,
$
and equality holds if and only if $(a,b,c)=(1,2,4)$.
\end{enumerate}

\item[\textup{(ii)}] \textbf{Case $k=2$.} One has
\[
F_{2}(c)\le 3\,F_{2}(a)F_{2}(b) \ \text{if and only if}\ c\le a+b,
\]
and equality holds if and only if $(a,b,c)=(2,2,4)$.

\item[\textup{(iii)}] \textbf{Case $k\ge 3$.} One has
\[
F_{k}(c)\le 3\,F_{k}(a)F_{k}(b) \ \text{if and only if}\ c<a+b,
\]
and equality occurs only at $(a,b,c)=(1,1,2)$.
\end{enumerate}
\end{lemma}
\medskip

\noindent
As remarked after \eqref{df:k_fibonacci}, the $k$-Fibonacci numbers form a Lucas-type sequence.
We will also make use of its companion, the \emph{$k$-Lucas sequence} (see, e.g., \cite[Section~5.1]{book}), defined by

\begin{equation}\label{df:k_lucas}
\left\{
\begin{aligned}
L_k(0) &= 2, \\
L_k(1) &= k, \\
L_k(n) &= k L_k(n-1) + L_k(n-2), \quad \text{for } n \ge 2.
\end{aligned}
\right.
\end{equation}
\medskip

For us, it will be relevant to study when $F_k(\ell)$ and $L_k(\ell)$ are divisible by $3$. This is the content of the following lemma. 

    \begin{lemma}\label{Lucascongmod4}
    Let $F_k(\ell)$ and $L_k(\ell)$ be the $\ell$-th $k$-Fibonacci and Lucas numbers, respectively. The following are true:
    \begin{itemize}
        \item 
        \text{If }$3 \nmid k$, $\begin{cases}
            3 \mid F_k(\ell) \text{ if and only if } 4 \mid \ell, \\
            3 \mid L_k(\ell) \text{ if and only if } \ell \equiv 2 \Mod{4}.
        \end{cases}$
        \item If $3 \mid k$, $\begin{cases}
            3 \mid F_k(\ell) \text{ if and only if } 2 \mid \ell, \\
            3 \mid L_k(\ell) \text{ if and only if } 2 \nmid \ell.
        \end{cases}$
    \end{itemize}
\end{lemma}

\begin{proof}
The statement for the $k$-Fibonacci numbers is an immediate consequence of general results on linear recurrence sequences modulo primes; see, for instance, \cite[Theorem~3]{Renault}. 
Indeed, applying that result with $(a,b)=(k,1)$ and $p=3$ yields $\alpha(3)=4$ when $3\nmid k$ and $\alpha(3)=2$ when $3\mid k$. 
The corresponding assertions for the $k$-Lucas numbers follow analogously.

\end{proof}

The following identity is classical in nature and relates $k$-Fibonacci numbers with $k$-Lucas numbers. A proof may be found for example, in \cite{book1} (as a special case of identity (2.14)), or, as a consequence of Theorem~2.4 in \cite{F}, together with equations~\eqref{eq:sum} and ~\eqref{eq:DOcagne}, to state two instances.

\begin{lemma}\label{lem:lucas}
For positive integers $a\geq b,$ the following identity holds:
\[
F_k(a+b) = F_k(a)\,L_k(b) - (-1)^b F_k(a-b).
\]
\end{lemma}

\section{Principal \texorpdfstring{$(2,k)$}{(2,k)}-Fibonacci branches}\label{section:2}

\subsection{Existence and positivity of $m$}

In this subsection, we study principal $(2,k)$-Fibonacci branches arising from Markoff $m$-triples of the form
$$
\left(a, F_k(n-r), F_k(n+r)\right).
$$
Our aim is to determine when such triples exist and when the corresponding Markoff parameter \(m\) is positive. For clarity, we distinguish between the cases \(3\mid k\) and \(3\nmid k\).

\subsubsection{\texorpdfstring{The case where $3 \mid k$}{The case where 3 divides k}}

In this case, we show that no such configurations can occur.

\begin{proposition}
Let $k$ and $m$ be positive integers with \(3 \mid k \) and $m> 0$. Then there are no branches of $m$-Markoff triples of the form \( (a, F_k(n-r), F_k(n+r)) \).
\end{proposition}

\begin{proof}
Suppose \( (a, F_k(n-r), F_k(n+r)) \) is a Markoff $m$-triple. Then it satisfies the identity
\[
m = a^2 + F_k(n-r)^2 + F_k(n+r)^2 - 3a F_k(n-r) F_k(n+r).
\]
By using Binet's formula \eqref{eqn:binet}, we compute
\begin{align*}
F_k(n \pm r)^2 &= \left( \frac{1}{\sqrt{k^2 + 4}} \right)^2 \left( \phi_k^{n \pm r} - \overline{\phi}_k^{\,\,n \pm r} \right)^2 \\
&= \frac{1}{k^2 + 4} \left( \phi_k^{2n \pm 2r} + \overline{\phi}_k^{\,\,2n \pm 2r} - 2(-1)^{n \pm r} \right), \\
F_k(n - r) F_k(n + r) &= \left( \frac{1}{\sqrt{k^2 + 4}} \right)^2 \left( \phi_k^{\,n-r} - \overline{\phi}_k^{\,\,n-r} \right) \left( \phi_k^{n+r} - \overline{\phi}_k^{\,\,n+r} \right) \\
&= \frac{1}{k^2 + 4} \left( \phi_k^{2n} + \overline{\phi}_k^{\,\,2n} - (-1)^{n-r}(\phi_k^{2r} + \overline{\phi}_k^{\,\,2r}) \right).
\end{align*}


Substituting these expressions in the Markoff identity, and multiplying through by \( (k^2 + 4) \), we obtain
\begin{align*}
(k^2 + 4) m &= (k^2 + 4) a^2 + \phi_k^{2n - 2r} + \overline{\phi}_k^{\,\,2n - 2r} + \phi_k^{2n + 2r} + \overline{\phi}_k^{\,\,2n + 2r} \\
&\quad - 2(-1)^{n - r} - 2(-1)^{n + r} - 3a \left( \phi_k^{2n} + \overline{\phi}_k^{\,\,2n} - (-1)^{n - r} (\phi_k^{2r} + \overline{\phi}_k^{\,\,2r}) \right).
\end{align*}
The right-hand side of this expression contains exponential terms in \( \phi_k^{2n} \) and \( \overline{\phi}_k^{\,\,2n} \), which grow unbounded as $n \to \infty$ unless their coefficients vanish. Since we are looking for solutions in a branch, there have to be infinitely many, and since $(k^2+4)m$ is bounded, it follows that the coefficients of \( \phi_k^{2n} \) and \( \overline{\phi}_k^{\,\,2n} \) must both vanish. In other words,
\[\phi_k^{2r} + \frac{1}{\phi_k^{2r}} - 3a = 0 \Longrightarrow \quad a = \frac{1}{3} \left( \phi_k^{2r} + \phi_k^{-2r} \right).
\]
Using Binet's formula again, we obtain:
\[
a = \frac{1}{3} \left( \phi_k^{2r} + \phi_k^{-2r} \right)
= \frac{1}{3} \left( (k^2 + 4) F_k(r)^2 + 2(-1)^r \right).
\]
Since $a$ is an integer, the numerator must be divisible by 3. Therefore, it follows that
\[
(k^2 + 4) F_k(r)^2 + 2(-1)^r \equiv 0 \pmod{3}.
\]
As we are assuming \( k \equiv 0 \pmod{3} \), it follows that \( k^2 + 4 \equiv 1 \pmod{3} \), and hence the previous identity reduces to
\begin{equation}
\label{eqn:congruencefibonacci}
F_k(r)^2  \equiv (-1)^r \pmod{3}.
\end{equation}
If $3 \mid k$, Lemma \ref{Lucascongmod4} implies that $F_k(r) \equiv 0 \pmod{3}$ if $2 \mid r$ and $F_k(r)^2 \equiv 1 \pmod{3}$ if $2 \nmid r$. This contradicts \eqref{eqn:congruencefibonacci}.

Therefore, in all cases, the condition for boundedness fails. Consequently, no such triple \( (a, F_k(n - r), F_k(n + r)) \) can form a branch of a Markoff $m$-triple when \( k \equiv 0 \pmod{3} \).
\end{proof}

\subsubsection{\texorpdfstring{The case where $3 \nmid k$}{The case where 3 does not divide k}}
We note that, as long as $3 \nmid k$, a similar argument to the one in Lemma \ref{Lucascongmod4} shows that, for each fixed $r$, the residue modulo $3$ of
\[
(F_{k}(r+1)+F_{k}(r-1))^2-2(-1)^r
\]
is independent of $k$. The resulting residues, which depend only on $r \bmod 4$, are listed in Table \ref{tab:mod3}.
\begin{table}[H]
\centering
\renewcommand{\arraystretch}{1.2}
\caption{Residues of $(F_{k}(r+1)+F_{k}(r-1))^{2}-2(-1)^{r} \pmod{3}$ when $3 \nmid k$}
\label{tab:mod3}
\begin{tabular}{|c|c|}
\hline
$r$ & $(F_{k}(r+1)+F_{k}(r-1))^{2}-2(-1)^{r}\pmod{3}$ \\
\hline
0 & 2 \\
1 & 0 \\
2 & 1 \\
3 & 0 \\
4 & 2 \\
5 & 0 \\
6 & 1 \\
7 & 0 \\
\hline
\end{tabular}
\end{table}

We define the rational number $\alpha_{k,r}$ by the expression
\begin{equation}\label{eq:deF_k(a)}
  \alpha_{k,r}:=\frac{(F_{k}(r+1)+F_{k}(r-1))^{2}-2(-1)^{r}}{3}.
\end{equation}

\begin{lemma}\label{lem:a_entero}
$\alpha_{k,r}$ is a positive integer if and only if $r$ is odd. 
\end{lemma}

\begin{proof}
We observe that $\alpha_{k,r}$ is an integer if and only if the numerator is divisible by 3. That is,
\[
(F_{k}(r+1) + F_{k}(r-1))^2 - 2(-1)^r \equiv 0 \pmod{3}.
\]

\noindent According to Table~\ref{tab:mod3}, this congruence holds if and only if $r$ is odd.
\end{proof}

The following lemma shows that $\alpha_{k,r}$, defined in \eqref{eq:deF_k(a)}, can be expressed both as the quotient of two $k$-Fibonacci numbers and in terms of a $k$-Lucas number.

\begin{lemma}\label{lemma:anitha} Let $\alpha_{k,r}$ be as given in \eqref{eq:deF_k(a)}, with $r$ odd. Then $$\alpha_{k,r}=\frac{F_{k}(4r)}{3F_{k}(2r)} = \frac{L_{k}(2r)}{3}.$$
\end{lemma}\begin{proof}
We begin by applying identity~\eqref{eq:sum} with \( a = b = 2r \), which gives
\begin{equation}
\label{eqn:fibonacci1}
  F_{k}(4r) = F_{k}(2r+1) F_{k}(2r) + F_{k}(2r) F_k(2r-1)
         = F_{k}(2r)(F_{k}(2r+1) + F_{k}(2r-1)).
\end{equation}
Dividing both sides by \( F_{k}(2r) \), we obtain
\[
  \frac{F_{k}(4r)}{F_{k}(2r)} = F_{k}(2r+1) + F_k(2r-1).
\]
\noindent Using identity~\eqref{eq:sum} again
\begin{align*}
  F_{k}(2r+1) &= F_k(r+1)^2 + F_k(r)^2 \quad \text{(with } a = r,\, b = r+1), \\
  F_k(2r-1) &= F_k(r)^2 + F_k(r-1)^2 \quad \text{(with } a = r-1,\, b = r).
\end{align*}
Thus,
\begin{equation}
\label{eqn:fibonacci2}
  F_{k}(2r+1) + F_{k}(2r-1) = F_{k}(r+1)^2 + 2F_{k}(r)^2 + F_{k}(r-1)^2.
\end{equation}
\noindent We now use the Simson identity~\eqref{eq:Simson} to express \( F_{k}(r)^2 \) as
\[
  F_{k}(r)^2 = F_{k}(r+1)F_{k}(r-1) - (-1)^r.
\]
Substituting into \eqref{eqn:fibonacci2}, we get
\begin{align*}
  F_{k}(2r+1) + F_{k}(2r-1)
  &= F_{k}(r+1)^2 + 2(F_{k}(r+1)F_{k}(r-1) - (-1)^r) + F_{k}(r-1)^2 \\
  &= F_{k}(r+1)^2 + 2F_{k}(r+1)F_{k}(r-1) + F_{k}(r-1)^2 - 2(-1)^r \\
  &= (F_{k}(r+1) + F_{k}(r-1))^2 - 2(-1)^r.
\end{align*}
\noindent Using  definition \eqref{eq:deF_k(a)}, we obtain
\[
  \frac{F_{k}(4r)}{F_{k}(2r)} = (F_{k}(r+1) + F_{k}(r-1))^2 - 2(-1)^r = 3\alpha_{k,r}.
\]
Therefore 
$\alpha_{k,r}=\frac{F_{k}(4r)}{3F_{k}(2r)}$, and this quotient is equal to  $\frac{L_{k}(2r)}{3}$ by Lemma \ref{lem:lucas}.
\end{proof}

We have already established that $r$ must be odd. 
We now prove that, for triples of the form $(\alpha_{k,r},F_k(n-r),F_k(n+r))$, 
the quantity $\mm(\alpha_{k,r},F_k(n-r),F_k(n+r))$ does not depend on $n$, except for a sign.

\begin{proposition}\label{thm:familia_par_impar}
Let $r$ be odd and $\alpha_{k,r}$ as in \eqref{eq:deF_k(a)}. Then, for all $n>r$,
\[
  \mm(\alpha_{k,r},F_{k}(n-r),F_{k}(n+r))=
  \alpha_{k,r}^{2}+(-1)^{n-r}F_{k}(r)^{2}(F_{k}(r+1)+F_{k}(r-1))^{2}.
\]
\end{proposition}

\begin{proof}
Let $r\ge 1$.  First, we recall the Catalan identity (see~\eqref{eq:Catalan}):
\begin{equation} \label{eq:Catalan2}
  F_{k}(n)^{2}-F_{k}(n-r)\,F_{k}(n+r)=(-1)^{\,n-r}\,F_{k}(r)^{2}.
\end{equation}
We add the identities \eqref{eq:sum} and \eqref{eq:DOcagne} (with $a =n$ and $b=r$ and $a=r$ and $b=n$, respectively) in order to eliminate the term $F_{k}(n+1)F_k(r)$. We get
\begin{equation}\label{eq:lineal}
  (F_{k}(r+1)+F_{k}(r-1))\,F_{k}(n)=F_{k}(n+r)+(-1)^{r}F_{k}(n-r).
\end{equation}

\noindent
Next, multiplying \eqref{eq:Catalan2} by 
\(\bigl(F_{k}(r+1)+F_{k}(r-1)\bigr)^{2}\) and substituting \(F_{k}(n)\) using~\eqref{eq:lineal}, we obtain
\begin{multline}\label{eq:Markoff_grande}
  F_{k}(n+r)^{2}+F_{k}(n-r)^{2}
  -\Bigl((F_{k}(r+1)+F_{k}(r-1))^{2}-2(-1)^{r}\Bigr)\,
   F_{k}(n+r)\,F_{k}(n-r)
  =\\(-1)^{\,n-r}\,F_{k}(r)^{2}\,\bigl(F_{k}(r+1)+F_{k}(r-1)\bigr)^{2}.
\end{multline}
Finally, we note that, by substituting \eqref{eq:deF_k(a)} into \eqref{eq:Markoff_grande}, it can be shown that the left-hand side of \eqref{eq:Markoff_grande} is $\mm(\alpha_{k,r}, F_{k}(n-r), F_{k}(n+r))-\alpha_{k,r}^2$, and hence the proposition follows.
\end{proof}

\medskip 

A priori, it is not clear whether the value of $\mm(\alpha_{k,r}, F_{k}(n-r), F_{k}(n+r))$ given in the previous proposition is positive and thus constitutes a  valid solution to the $m-$Markoff equation with $m>0$. This is the content of the following proposition.

\begin{proposition}\label{prop2.5}
Let $r$ be odd and $n$ even, with $n > r>0$. Then
\[
\mm\left(\alpha_{k,r}, F_k(n-r), F_k(n+r)\right) > 0
\quad \text{if and only if} \quad
k \geq 3,
\]
and
\[
\mm\left(\alpha_{k,r}, F_k(n-r), F_k(n+r)\right) =0 
\quad \text{if and only if} \quad
1 \le k \le 2 \text{ and } r = 1.
\]
\end{proposition}

\begin{proof}
By Proposition~\ref{thm:familia_par_impar},
\[
\mm\left(\alpha_{k,r},F_k(n-r),F_k(n+r)\right)
=\alpha_{k,r}^{2}-F_k(r)^{2}\bigl(F_k(r+1)+F_k(r-1)\bigr)^{2}.
\]
Since $\alpha_{k,r}>0$ (Lemma~\ref{lem:a_entero}), the inequality
$\mm(\alpha_{k,r},F_k(n-r),F_k(n+r))>0$ is equivalent to
\[
\alpha_{k,r}>F_k(r)\bigl(F_k(r+1)+F_k(r-1)\bigr).
\]
As $r$ is odd, by Lemma \ref{lem:a_entero}, equation \eqref{eq:deF_k(a)} gives
\[
\alpha_{k,r}=\frac{\bigl(F_k(r+1)+F_k(r-1)\bigr)^2+2}{3},
\]
so the previous inequality is equivalent to
\[
F_k(r+1)+F_k(r-1)+\frac{2}{F_k(r+1)+F_k(r-1)}>3F_k(r).
\]
This holds for $k\ge3$ because $F_k(r+1)\ge kF_k(r)\ge 3F_k(r)$ by
\eqref{eq:basic_bound}, and the extra fraction is positive.
For $k\in\{1,2\}$ and $r\ge3$, one has
$F_k(r+1)+F_k(r-1)\le 3F_k(r)-1$ and $F_k(r+1)+F_k(r-1)\ge 4$, hence
\[
F_k(r+1)+F_k(r-1)+\frac{2}{F_k(r+1)+F_k(r-1)}
\le 3F_k(r)-1+\frac12<3F_k(r),
\]
so the inequality fails. Therefore
$\mm(\alpha_{k,r},F_k(n-r),F_k(n+r))>0$ holds exactly when $k\ge3$.

Finally, by Proposition~\ref{thm:familia_par_impar} again,
\[
\mm\left(\alpha_{k,r},F_k(n-r),F_k(n+r)\right)=0
\]
is equivalent to
\[
\alpha_{k,r}=F_k(r)\bigl(F_k(r+1)+F_k(r-1)\bigr).
\]
By Lemma~\ref{lemma:anitha} we also have $\alpha_{k,r}=F_k(4r)/(3F_k(2r))$, hence
the previous equality is equivalent to
\[
F_k(4r)=3F_k(2r)\,F_k(r)\bigl(F_k(r+1)+F_k(r-1)\bigr).
\]
Using \eqref{eq:sum} with $(a,b)=(2r,2r)$ gives
$F_k(4r)=F_k(2r)\bigl(F_k(2r+1)+F_k(2r-1)\bigr)$, and using \eqref{eq:sum} with
$(a,b)=(r,r)$ and $(a,b)=(r-1,r)$ gives
$F_k(2r)=F_k(r)\bigl(F_k(r+1)+F_k(r-1)\bigr)$. Cancelling the common factor
$F_k(2r)>0$, we get the equivalent condition
\[
F_k(2r+1)+F_k(2r-1)=3F_k(2r).
\]
With the recurrence $F_k(2r+1)=kF_k(2r)+F_k(2r-1)$ this is equivalent to
\[
(3-k)F_k(2r)=2F_k(2r-1),
\]
which forces $k\le2$, and for $k\in\{1,2\}$ it forces $r=1$ because
$F_k(2r)=kF_k(2r-1)+F_k(2r-2)>2F_k(2r-1)$ for $r\ge2$.
Conversely, if $r=1$ and $k\in\{1,2\}$, then $(3-k)F_k(2)=2F_k(1)$ holds, so
$\mm(\alpha_{k,1},F_k(n-1),F_k(n+1))=0$.
\end{proof}

\subsection{\texorpdfstring{Classification of non-minimal triples with two $k$-Fibonacci components}{Non-minimal triples with two k-Fibonacci components}}

In this subsection, we prove Theorem \ref{the:nonminimal_general}, showing that every non-minimal Markoff \( m \)-triple with two $k$-Fibonacci components of the form \( (a, F_k(b), F_k(c)) \) corresponds to parameters \( a = \alpha_{k,r} \) as defined in~\eqref{eq:deF_k(a)}, \( b = n - r \), and \( c = n + r \), for some odd integer \( r \).
\medskip 

\begin{lemma}[Auxiliary bound]
\label{lem:n>=l}
Let $(a,F_k(n),F_k(n+\ell))$ be an ordered, non-minimal Markoff $m$-triple with $m>0$
. Then $n \ge \ell$.
\end{lemma}

\begin{proof}
Since the triple is ordered and non-minimal, Remark \ref{rmk:usefulnotes} gives that $a_q < a \le F_k(n)$, where
\[
a_q = \frac{F_k(n+\ell)}{3F_k(n)}.
\]
In particular, $a_q < F_k(n)$, so that $F_k(n+\ell) < 3F_k(n)^2$. By Lemma~\ref{lemma:F(c)<=3F(n)F(m)}, we know that
$
F_k(n+\ell) < 3F_k(n)^2$ if and only if $n + n \ge n+\ell,$
which simplifies to $n \ge \ell$.
\end{proof}
\medskip 
\begin{lemma}
\label{lemma:nonminimalFib}
   Suppose that $(a, F_k(n), F_k(n+\ell))$ is an ordered, non-minimal Markoff triple with two $k$-Fibonacci components and assume that $3 \nmid k$. Then, either $k^\ell < 4$ or 
    $\ell \equiv 2 \pmod{4}$ and $a = L_{k}(\ell)/3$.
\end{lemma}

\begin{proof}
Recall that by applying the definition in Section \ref{sec:nonminimals} to the tuple $(a,F_k(n),F_k(n+\ell))$, we have that 
\[a_q = \frac{F_k(n+\ell)}{3F_k(n)}. \]
We shall make use of the identity
\[
F_k(i+j) = F_k(i) L_k(j) - (-1)^j F_k(i-j),
\]
where $L_k(j)$ denotes the associated $k$-Lucas sequence (see Lemma~\ref{lem:lucas}). In particular, taking $i = n$ and $j = \ell$, we obtain
    \begin{equation}
        \label{eqn:lucas}
    a_q = \frac{L_{k}(\ell)}{3} - (-1)^\ell\frac{F_{k}(n-\ell)}{3F_k(n)}.
    \end{equation}
    Note, that by induction on \eqref{eq:basic_bound}, it is elementary to show that 
    \begin{equation}
        \label{eqn:elementarybound}
    \frac{F_{k}(n-\ell)}{F_k(n)} \le \frac{1}{k^\ell} \quad \text{and} \quad \frac{b}{c} = \frac{F_k(n)}{F_k(n+\ell)} \le \frac{1}{k^\ell}.
    \end{equation}
    We distinguish two cases. Firstly, if $\ell$ is even, then $a_q <\frac{L_{k}(\ell)}{3}$. If $\frac{L_{k}(\ell)}{3} < a_p$, Proposition \ref{prop:smallerinterval} implies that 
    \begin{equation}
        \label{eqn:inequalitieskl}
    a_p - a_q < \frac{1}{k^\ell} \quad \text{ and so } \quad \left|a_p-\frac{L_{k}(\ell)}{3}\right| < \frac{1}{k^\ell}\quad \text{and} \quad \left|\frac{L_{k}(\ell)}{3}-a_q\right| < \frac{1}{k^\ell}.
    \end{equation}
    Suppose now that $a_q < a_p \le \frac{L_{k}(\ell)}{3}$. Then, we see that 
    \[\frac{L_{k}(\ell)}{3}-a_p\le \frac{L_{k}(\ell)}{3}-a_q < \frac{1}{k^\ell}\]
    by \eqref{eqn:elementarybound} and so \eqref{eqn:inequalitieskl} also holds.

    Finally, if $\ell$ is odd, we have that $\frac{L_{k}(\ell)}{3} < a_q < a_p$ and Proposition \ref{prop:smallerinterval} together with \eqref{eqn:lucas} and \eqref{eqn:elementarybound}, it follows that 
    \[a_p - \frac{L_{k}(\ell)}{3} < \frac{4}{3k^\ell}.\]
    If we suppose that $k^\ell \ge 4$, it readily follows that, in either of the aforementioned cases, we have that 
    \begin{equation}
        \label{eqn:mindistance}
    \max\left\{\left|a_p-\frac{L_{k}(\ell)}{3}\right|, \left|\frac{L_{k}(\ell)}{3}-a_q\right|\right\} < \frac{1}{3}.
    \end{equation}
    Now, we note that, since $3 \nmid k$, $L_{k}(\ell)$ is only divisible by $3$ if $\ell \equiv 2\pmod{4}$ (see Lemma \ref{Lucascongmod4}). Suppose first that $\ell \not \equiv 2 \pmod{4}$. Then, the distance from $L_{k}(\ell)$ to its closest integer is precisely $1/3$. From \eqref{eqn:mindistance}, we deduce that the interval $[a_q,a_p]$ contains no integer.

    Consequently, it follows that $\ell \equiv 2 \pmod{4}$, so that $\ell$ is even. In this case, $\frac{L_{k}(\ell)}{3}$ is the unique integer in the interval $(a_q, a_p)$ and then $a = \frac{L_{k}(\ell)}{3}$ by Corollary \ref{propr:unicity}.

\end{proof}

\begin{lemma}\label{lemma_4.2}
    If $k^\ell \ge 4$ and $3 \mid k$, there are no non-minimal Markoff $m$-triples of the form $(a, F_k(n), F_{k}(n+\ell) )$
\end{lemma}

\begin{proof}
    By an identical argument to the previous lemma, we show that \eqref{eqn:mindistance} still holds if $3 \mid k$. However, now $L_{k}(\ell)$ is divisible by $3$ if and only if $\ell$ is odd and so (by mimicking the case $l \not \equiv 2 \pmod{4}$ in the proof of Lemma \ref{lemma:nonminimalFib}) it follows that the only possible ordered non-minimal solutions are of the form $(a, F_k(n), F_{k}(n+\ell))$ with $\ell$ odd.
    
    However, then \eqref{eqn:lucas} implies that $L_{k}(\ell)/3 < a_q < a_p$ and, since $L_{k}(\ell)/3$ is an integer and $a_p-L_{k}(\ell)/3 < 1/3$, it follows that the interval $(a_q, a_p)$ contains no integer. Consequently, there is no possible value of $a$ and thus there is no non-minimal Markoff $m$-triple.
\end{proof}
\medskip 

We now address the first case in Lemma \ref{lemma:nonminimalFib}, namely when $k^l < 4$. The next lemma characterises all non-minimal Markoff triples in this range.

\begin{lemma}
\label{lemma:klsmall}
    Let $2 \le k \le 3$. Then, the only possible ordered non-minimal Markoff triples $(a, F_k(n), F_{k}(n+1))$ with $\mm(a, F_k(n), F_{k}(n+1))\ge0$ correspond to $m=0$, $k=2$ and the triples $(1,F_{k}(1),F_{k}(1)), (1,F_{k}(1),F_{k}(2))$ and $(1,F_{k}(2),F_{k}(3))$.
\end{lemma}

\begin{proof}

Let $(a, F_k(n), F_{k}(n+1))$ be an ordered non-minimal Markoff triple with $a \neq 1$. Since $\mm(a, F_k(n), F_{k}(n+1))\ge0$, it follows that 
\begin{equation}
\label{eqn:auxineq}
    a^2 + F_k(n)^2 + F_{k}(n+1)^2 \ge 3aF_k(n)F_{k}(n+1) \ge \frac{3aF_{k}(2n)}{1+\frac{1}{k^2}},
\end{equation}
where the last inequality follows by \eqref{eq:bound_product2}. Since $2 \le k \le 3$, it follows that $1+1/k^2 \le 5/4$ and so \eqref{eqn:auxineq} implies that 
\begin{equation}
    \label{eqn:newuglybound}
a^2+F_k(n)^2+F_{k}(n+1)^2 \ge \frac{12}{5}aF_{k}(2n).
\end{equation}

We note that, by applying \eqref{eq:Catalan} with $r = n+1$, we get that
\[F_k(n)^2 + F_{k}(n+1)^2 = F_{k}(2n+1)F_{k}(-1) = F_{k}(2n+1).\]
Combining this with \eqref{eqn:newuglybound} and the fact that $F_{k}(2n+1) \le (k+1) F_{k}(2n) \le 4F_{k}(2n)$, we see that
\[{a^2+4F_{k}(2n) \ge}a^2 + F_{k}(2n+1) \ge \frac{12}{5}aF_{k}(2n),\]
or equivalently,
\[F_{k}(2n)\left(\frac{12}{5}a-4\right) \le a^2.\]
Since $a >1$ by assumption, this inequality can be written as
\begin{equation}
    \label{eqn:ainequality}
F_{k}(2n) \le \frac{a^2}{\frac{12}{5}a-4} \le a+3.
\end{equation}
If $a \ge 2$, the triple $(a, F_k(n), F_{k}(n+1))$ can only be ordered if $n \ge 2$. In this case, we have $F_{k}(2n) = kF_{k}(2n-1)+F_{k}(2n-2) > kF_{k}(2n-1) \ge 2 F_{k}(2n-1)$. Combining this inequality with \eqref{eqn:ainequality}, we get that 
\begin{equation}
    \label{eqn:contradictioninequality}
2F_{k}(2n-1)\le a+3 \le F_k(n)+3,
\end{equation}
where the last inequality follows from the fact that the triple $(a, F_k(n), F_{k}(n+1))$ is ordered by assumption. However, \eqref{eqn:contradictioninequality} is never satisfied, which proves the lemma for $a > 1$.

Finally, let us suppose that $a = 1$. Since $F_{k}(1) = 1$, the triple $(1, F_k(n), F_{k}(n+1))$ consists of three $k$-Fibonacci numbers. If $m>0$, \cite[Theorem 1.1]{ACMRS25b} implies that this triple is minimal, contradicting our original assumption.

    The only remaining case to consider is when $a=1$ and $m=0$. Since the triple $(1, F_k(n), F_k(n+1))$ is not minimal, then 
    $F_k(n+1)<3 F_k(n)\,.$
    By the recurrence relation for the $k$-Fibonacci sequence $kF_k(n)\leq F_k(n+1)$, then $k\,F_k(n)\leq 3F_k(n)$ $k$ which forces $k < 3$. 
Since by hypothesis $2 \le k \le 3$, it follows that $k = 2$. Therefore, the triple $(1, F_2(n), F_2(n+1))$ consists of three consecutive Pell numbers (recall that $P_1 = 1$). 
According to \cite[Theorem~1.1]{KST}, the only solutions to the corresponding Pell--Markoff equation are
\[
(1, 1, 1), \quad (1, 1, 2), \quad \text{and} \quad (1, 2, 5).
\]
All of these are ordered and non-minimal, and they coincide precisely with the triples listed in the statement of the lemma.
\end{proof}

\subsubsection{The Fibonacci case}
After Lemma \ref{lemma:klsmall}, in order to fully characterise triples\\ $(a, F_k(n), F_{k}(n+\ell))$ with $k^\ell < 4$, it only remains to consider the case $k=1$, which corresponds to classical Fibonacci numbers.

\begin{lemma}
\label{lem:k=1_l_odd}
Assume that $(a,F_1(n),F_1(n+\ell))$ is an ordered, non-minimal Markoff $m$-triple with $m>0$. If $\ell\ge 3$ is odd, then no such triple exists.
\end{lemma}

\begin{proof}
By Lemma~\ref{lem:n>=l}, we have $n\ge \ell$. Since $\ell$ is odd, identity~\eqref{eqn:lucas} gives
\[
a_q=\frac{L_1(\ell)}{3}+\frac{F_1(n-\ell)}{3F_1(n)}.
\]
Moreover, Proposition~\ref{prop:smallerinterval} yields
\[
a_p-a_q<\frac{F_1(n)}{F_1(n+\ell)}.
\]
Since the triple is non-minimal, we have $a_q<a<a_p$, and therefore
\[
0<a-\frac{L_1(\ell)}{3}
<
\frac{F_1(n-\ell)}{3F_1(n)}+\frac{F_1(n)}{F_1(n+\ell)}.
\]
We now show that the right-hand side is always strictly smaller than $1/3$.

Since $\ell\ge 3$ is odd and $n\ge \ell$, we have $n\ge 3$. Also, by \cite[Table~2]{ACMRS25a},
\[
\frac{F_1(n)}{F_1(n+\ell)}\le \frac14.
\]

First suppose that $n-\ell\le 2$. If $(n,\ell)\neq (3,3),(4,3)$, then necessarily $n\ge 5$, and hence
\[
\frac{F_1(n-\ell)}{3F_1(n)}\le \frac{1}{3F_1(n)}\le \frac1{15}.
\]
Therefore,
\[
\frac{F_1(n-\ell)}{3F_1(n)}+\frac{F_1(n)}{F_1(n+\ell)}
\le \frac1{15}+\frac14<\frac13.
\]

Now suppose that $n-\ell>2$. Then again by \cite[Table~2]{ACMRS25a},
\[
\frac{F_1(n-\ell)}{F_1(n)}\le \frac14.
\]
Hence
\[
\frac{F_1(n-\ell)}{3F_1(n)}\le \frac1{12}.
\]
Since $\ell\ge 3$, we also have
\[
\frac{F_1(n)}{F_1(n+\ell)}<\frac14,
\]
and thus
\[
\frac{F_1(n-\ell)}{3F_1(n)}+\frac{F_1(n)}{F_1(n+\ell)}
<
\frac1{12}+\frac14=\frac13.
\]

We conclude that
\[
0<a-\frac{L_1(\ell)}{3}<\frac13.
\]
On the other hand, when $\ell$ is odd, Lemma~\ref{Lucascongmod4} implies that
\[
L_1(\ell)\equiv 1 \pmod 3,
\]
so $L_1(\ell)/3$ is at distance exactly $1/3$ from the nearest integer. Since $a$ is an integer, this is impossible.

It remains to check the exceptional cases $(n,\ell)=(3,3)$ and $(4,3)$. These correspond to the triples
\[
(a,F_1(3),F_1(6))=(a,2,8)
\qquad\text{and}\qquad
(a,F_1(4),F_1(7))=(a,3,13).
\]
A direct computation shows that neither of them yields a non-minimal Markoff $m$-triple with $m>0$. This proves the lemma.
\end{proof}
\medskip

\begin{lemma}\label{lemma:k=l=1}
There is no ordered non-minimal Markoff $m$-triple ($m>0$) of the form
$$
\bigl(a,\,F_{1}(n),\,F_{1}(n+1)\bigr).
$$ 
\end{lemma}

\begin{proof}
Assume that such a triple exists and write $x:=F_{1}(n)$ and $y:=F_{1}(n+1)$, so $y\le 2x$ and $y\ge x\ge 1$.  
From the Markoff equation \(a^2+x^2+y^2=3axy+m\), taking the smaller root in $a$ (as the triple is ordered), we obtain
\[
a = \frac{3xy - \sqrt{9x^2y^2 - 4x^2 - 4y^2 + 4m}}{2}.
\]

Since we assume that \(m > 0\), the discriminant of the previous quadratic increases with \(m\), so the expression for \(a\) decreases. Therefore, we define
\[
A(x,y) := \frac{3xy - \sqrt{9x^2y^2 - 4x^2 - 4y^2}}{2},
\]
and conclude that
\[
a < A(x,y).
\]

We claim that \(A(x,y)\le 1\), for all consecutive Fibonacci numbers.  
Indeed, $A(x,y)\le 1$ if and only if 
$$3xy-2\le\sqrt{9x^2y^2-4x^2-4y^2}
$$
and squaring (both sides are positive since $x,y\ge 1$) and manipulating the equations gives that
\[
x^2+y^2-3xy+1\le 0.
\]
Since \(x\ge 1\) and \(x\le y\le 2x\), then
\[x^2+y^2-3xy+1 = (x-y)^2-xy+1 \le x^2-xy+1\le0.\]
The left-hand side is non-positive, thereby proving the claim. 

Thus \(a< A(x,y) \le 1\), a contradiction.

\end{proof}
\medskip 

\begin{lemma}
\label{lem:k=1_l_congruent with 0 modulus 4}
There are no ordered non-minimal Markoff $m$-triples $(a,F_1(n),F_1(n+\ell))$ with $m>0$, for any $\ell \equiv 0 \pmod{4}$.
\end{lemma}

\begin{proof}
Assume that such a  triple exists. For $n>2$ we have 
$$
F_1(n) \;=\; F_1(n-1) + F_1(n-2) \;\ge\; 2F_1(n-2).
$$
Iterating this inequality yields
$$
F_1(n) \;\ge\; 2^{\ell/2}\,F_1(n-\ell), \qquad \text{for } n>\ell.
$$
Consequently,
$$
\frac{F_1(n-\ell)}{F_1(n)} \;\le\; \frac{1}{2^{\ell/2}}.
$$

This implies, applying the argument of the proof of Lemma \ref{lemma:nonminimalFib} and taking into account that $\ell$ is an even number, $\ell \geq 4$, that 
\begin{equation}
        \label{eqn:mindistance adapted}
    \max\left\{\left|a_p-\frac{L_{1}(\ell)}{3}\right|, \left|\frac{L_{1}(\ell)}{3}-a_q\right|\right\} \leq \frac{1}{2^{\frac{\ell}{2}}}< \frac{1}{3}.
    \end{equation}

    According to the reasoning of the proof of Lemma \ref{lemma:nonminimalFib}, this yields $\ell\equiv 2 \pmod{4}$, a contradiction with the assumption.    
\end{proof}

\subsubsection{The general case} With the preparatory results established, we are now in a position to state the main theorem of this section, which provides a complete characterisation of ordered non-minimal triples whose last two components are $k$-Fibonacci numbers.

\begin{lemma}\label{lem:ordered_nonminimal_boundary}
Let $r$ be odd and $k \ge 1$. Consider the triple
\[
\left(\frac{F_{k}(4r)}{3F_{k}(2r)},\,F_{k}(N-r),\,F_{k}(N+r)\right).
\]
Then the following hold:
\begin{enumerate}
    \item The triple is ordered if and only if
    \[
    \begin{cases}
    N>3r, & \text{if } k>2,\\[3pt]
    N\ge 3r, & \text{if } k\in\{1,2\}.
    \end{cases}
    \]
    \item If the triple is ordered, then it is minimal if and only if $N=3r$, and non-minimal if and only if $N>3r$.
\end{enumerate}
\end{lemma}

\begin{proof} We first determine when the triple is ordered. By monotonicity of the $k$-Fibonacci sequence, 
\[ F_k(N-r)\le F_k(N+r), \qquad\text{for all } N>r. \] 
Thus, it suffices to compare the first two entries. Since 
$\frac{F_k(4r)}{3F_k(2r)}=\frac{L_k(2r)}{3},$ orderedness is equivalent to 
\[ L_k(2r)\le 3F_k(N-r). \] 
If $k>2$, this inequality holds for all $N>3r$, while it fails at $N=3r$, since then it becomes \[ L_k(2r)\le 3F_k(2r), \] which is false for all $r\ge 1$. Indeed, for $r=1$,
\[ L_k(2)=k^2+2>3k=3F_k(2), \qquad (k>2), \] 
and the inequality $L_k(2r)>3F_k(2r)$
is preserved by the recurrence relations \eqref{df:k_fibonacci} and \eqref{df:k_lucas}. If $k\in\{1,2\}$, the inequality already holds at $N=3r$, since 
\[ L_k(2r)\le 3F_k(2r) \qquad\text{for all } r\ge 1. \] 
Therefore, the triple is ordered if and only if 
\[ \begin{cases} N>3r, & \text{if } k>2,\\[3pt] N\ge 3r, & \text{if } k\in\{1,2\}. \end{cases} \] 
Assume now that 
\[ (x,y,z)=\left(\frac{F_k(4r)}{3F_k(2r)},\,F_k(N-r),\,F_k(N+r)\right) \] 
is ordered. Since minimality is equivalent to \(z\ge 3xy\), and 
$3xy=\frac{F_k(4r)}{F_k(2r)}\,F_k(N-r),$ the triple is minimal if and only if 
\[ F_k(N+r)\ge \frac{F_k(4r)}{F_k(2r)}\,F_k(N-r). \] 
Using Lemma~\ref{lem:lucas} with \(a=b=2r\), namely
\[ F_k(4r)=F_k(2r)L_k(2r), \]
this is equivalent to 
\[ F_k(N+r)\ge L_k(2r)\,F_k(N-r). \] 
Using $L_k(2r)=F_k(2r+1)+F_k(2r-1)$ and the addition formula \eqref{eq:sum}, we get 
\begin{align*} L_k(2r)F_k(N-r)-F_k(N+r) &=F_k(2r-1)F_k(N-r)-F_k(2r)F_k(N-r-1). \end{align*}
Therefore, the sign is determined by \[ \frac{F_k(2r-1)}{F_k(2r)}-\frac{F_k(N-r-1)}{F_k(N-r)}. \] 

Since the sequence \(\frac{F_k(2n-1)}{F_k(2n)}\) is strictly decreasing in \(n\) and converges to \(\frac{1}{\phi_k}\), whereas the sequence \(\frac{F_k(2n)}{F_k(2n+1)}\) is strictly increasing in \(n\) and also converges to \(\frac{1}{\phi_k}\), it follows that, if \(N\) is odd and \(N-r>2r\), then
\[
\frac{F_k(N-r-1)}{F_k(N-r)}<\frac{F_k(2r-1)}{F_k(2r)},
\]
so the difference is positive. If \(N\) is even, then
\[
\frac{F_k(2r-1)}{F_k(2r)}-\frac{F_k(N-r-1)}{F_k(N-r)}
>
\frac{1}{\phi_k}-\frac{1}{\phi_k}=0.
\]
Therefore,
\[
\frac{F_k(2r-1)}{F_k(2r)}-\frac{F_k(N-r-1)}{F_k(N-r)}
\begin{cases}
=0, & \text{if } N-r=2r,\\[3pt]
>0, & \text{if } N-r>2r.
\end{cases}
\]
Hence,
\[
L_k(2r)F_k(N-r)-F_k(N+r)
\begin{cases}
=0, & \text{if } N=3r,\\[3pt]
>0, & \text{if } N>3r.
\end{cases}
\]
Equivalently,
\[
F_k(N+r)
\begin{cases}
=3xy, & \text{if } N=3r,\\[3pt]
<3xy, & \text{if } N>3r.
\end{cases}
\]
Hence, the ordered triple is minimal if and only if \(N=3r\), and non-minimal if and only if \(N>3r\).
\end{proof}

\begin{theorem}[\textbf{Theorem \ref{the:nonminimal_general} of the Introduction}]
Let $(a,b,c)$ be a non-minimal Markoff $m$-triple with $m>0$. Then \(b\) and \(c\) are $k$-Fibonacci numbers if and only if
\[
(a,b,c)=\left(\alpha_{k,r},\,F_k(N-r),\,F_k(N+r)\right),
\]
where $\alpha_{k,r}=\frac{F_k(4r)}{3F_k(2r)},$ \(r\) is an odd integer, \(3\nmid k\), and \(N>3r\); moreover, if \(k\in\{1,2\}\), then \(N\) must be odd.
\end{theorem}

\begin{proof}
Assume first that we are given an ordered, non-minimal Markoff $m$-triple with $m>0$, whose last two entries are $k$-Fibonacci numbers, that is, $
(a, F_k(n), F_k(n+\ell)).$
When $k=1$, Lemmas~\ref{lem:k=1_l_odd}, \ref{lemma:k=l=1}, and \ref{lem:k=1_l_congruent with 0 modulus 4} together imply that $\ell \equiv 2 \pmod{4}$.  
If instead $k\ge 2$ but $k^\ell<4$, Lemma~\ref{lemma:klsmall} again excludes the existence of non-minimal triples with $m>0$.  
Consequently, any such triple must satisfy $k^\ell \ge 4$ and $3\nmid k$, by Lemma~\ref{lemma_4.2}.

Lemma~\ref{lemma:nonminimalFib} then yields $\ell\equiv 2\pmod{4}$ and $a=L_{k}(\ell)/3$. Writing $\ell=2r$, Lemma~\ref{Lucascongmod4} implies $r$ is odd (so $L_k(2r)/3\in\mathbb{Z}$). Lemma~\ref{lem:n>=l} ensures $n\ge \ell$, so set $N:=n+r$, giving $(F_k(n),F_k(n+\ell))=(F_{k}(N-r),F_{k}(N+r))$. Finally, $L_k(2r)=F_{k}(4r)/F_{k}(2r)$ (Lemma~\ref{lem:lucas}), hence $a=F_{k}(4r)/(3F_{k}(2r))$. Now, Lemma~\ref{lem:ordered_nonminimal_boundary} provides the conditions on $N$ and $r$, and the assumption $m>0$ implies, by Proposition~\ref{prop2.5}, that $N$ must be odd whenever $k<3$.

Assume now that we are given a triple
$\left(\frac{F_{k}(4r)}{3F_{k}(2r)},\,F_{k}(N-r),\,F_{k}(N+r)\right),$
satisfying the conditions of the statement. Since $r$ is odd and $k$ is not
divisible by $3$, the triple has integer entries. The conditions on $N$ and
$r$ imply that the triple is ordered and non-minimal by
Lemma~\ref{lem:ordered_nonminimal_boundary}. Finally, the positivity
condition $m>0$ follows from Proposition~\ref{prop2.5}: it holds for all
$k>2$, and for $k\in\{1,2\}$ precisely when $N$ is odd. This completes the
classification.

\end{proof}

\section{Distribution of the principal $(2,k)$-Fibonacci branches in trees}
\label{section:5}

In this section, we study how the principal $(2,k)$-Fibonacci branches are distributed among Markoff trees. As a first step, we show that consecutive triples in these branches are related by the Vieta transformations defined in \eqref{def:vieta}.

\begin{lemma}
\label{lemma:enganche}
Let \(r\ge1\) be odd and assume \(3\nmid k\).
\begin{enumerate}[label=\textup{(\alph*)}]
  \item If \(k>2\) and \(\ell>2r\), or if \(k\in\{1,2\}\), \(\ell> 2r\) and \(\ell\) is even,
  then the triple
  \[
  \bigl(\alpha_{k,r},F_k(\ell),F_k(\ell+2r)\bigr)
  \]
  is a Markoff \(m\)-triple with \(m>0\). Moreover,
  \[
    \nu_2\!\bigl(\alpha_{k,r},F_k(\ell),F_k(\ell+2r)\bigr)
    =\bigl(\alpha_{k,r},F_k(\ell+2r),F_k(\ell+4r)\bigr).
  \]

  \item If \(k\in\{1,2\}\) and \(\ell<2r\), or if \(k>2\) and \(\ell\le 2r\), then
  \[
    \nu_1\!\bigl(F_k(\ell),\alpha_{k,r},F_k(\ell+2r)\bigr)
    =\bigl(\alpha_{k,r},F_k(\ell+2r),F_k(\ell+4r)\bigr).
  \]
\end{enumerate}
\end{lemma}

\begin{proof}
Set \(N=\ell+r\). Then
$
(\alpha_{k,r},F_k(\ell),F_k(\ell+2r))
=
(\alpha_{k,r},F_k(N-r),F_k(N+r)).
$ Assume first that \(k>2\) and \(\ell>2r\), or that \(k\in\{1,2\}\), \(\ell\ge 2r\) and \(\ell\) is even.
Since \(r\) is odd, the condition \(\ell>2r\) is equivalent to \(N>3r\), and
\(\ell\ge 2r\) is equivalent to \(N\ge 3r\). Hence, by
Lemma~\ref{lem:ordered_nonminimal_boundary}, the triple
$(\alpha_{k,r},F_k(\ell),F_k(\ell+2r))$ is ordered and non-minimal. Moreover, by Proposition~\ref{prop2.5}, its
Markoff parameter satisfies \(m>0\), when
\(k\in\{1,2\}\), and for all \(\ell>2r\) when \(k>2\).

We now prove the Vieta identities. Using the generalisation of Vajda’s identity
(Lemma~\ref{lemma:Vajda}) with \(n=2r\), \(a=2r\) and \(b=\ell\), we get
\[
F_k(4r)F_k(\ell+2r)-F_k(2r)F_k(\ell+4r)
=
(-1)^{2r}F_k(2r)F_k(\ell)
=
F_k(2r)F_k(\ell).
\]
Dividing by \(F_k(2r)\), and using Lemma~\ref{lemma:anitha},
\[
\frac{F_k(4r)}{F_k(2r)}F_k(\ell+2r)-F_k(\ell+4r)=F_k(\ell),
\]
that is,
\[
3\alpha_{k,r}F_k(\ell+2r)-F_k(\ell)=F_k(\ell+4r).
\]
Therefore,
\[
\nu_2(\alpha_{k,r},F_k(\ell),F_k(\ell+2r))
=
(\alpha_{k,r},F_k(\ell+2r),3\alpha_{k,r}F_k(\ell+2r)-F_k(\ell))
=
(\alpha_{k,r},F_k(\ell+2r),F_k(\ell+4r)),
\]
which proves part \textup{(a)}.

For part \textup{(b)}, we apply the same identity:
\[
\nu_1(F_k(\ell),\alpha_{k,r},F_k(\ell+2r))
=
(\alpha_{k,r},F_k(\ell+2r),3\alpha_{k,r}F_k(\ell+2r)-F_k(\ell))
=
(\alpha_{k,r},F_k(\ell+2r),F_k(\ell+4r)).
\]
This proves \textup{(b)}.
\end{proof}

\noindent We recall that given a Markoff $m$-triple $(a_0,b_0,c_0)$, we write
\[
\BB(a_0,b_0,c_0):=\{\nu_2^{\,n}(a_0,b_0,c_0)\mid n\in\ZZ_{\ge0}\}
\]
for the \emph{branch} rooted at $(a_0,b_0,c_0)$.
We also recall that a  Markoff $m$-triple $(a,b,c)$ is minimal if
$$\Phi(a,b,c):=c-3ab\geq0.$$

\begin{theorem}[\textbf{Theorem \ref{prop:branches} of the Introduction}]
Fix an odd integer \(r\ge1\) and assume \(3\nmid k\).
Then the family of triples
\[
\bigl(\alpha_{k,r},\,F_k(\ell),\,F_k(\ell+2r)\bigr), \qquad \ell\in\NN,
\]
with
\[
\begin{cases}
\ell>2r, & \text{if } k>2,\\[3pt]
\ell \geq 2r \text{ and }\ell\ \text{even}, & \text{if } k\in\{1,2\},
\end{cases}
\]
is distributed among exactly \(2r\) distinct branches of \(m\)-trees, described as follows:
\begin{enumerate}
\item For each even \(\ell_0\in\{2,4,\dots,2r\}\), the triple
\((F_k(\ell_0),\alpha_{k,r},F_k(\ell_0+2r))\) is minimal. The corresponding branch is
\[
\BB\big(\nu_1(F_k(\ell_0),\alpha_{k,r},F_k(\ell_0+2r))\big)
=
\BB\big(\alpha_{k,r},F_k(\ell_0+2r),F_k(\ell_0+4r)\big).
\]

\item For each odd \(\ell_0\in\{1,3,\dots,2r-1\}\) and \(k\ge4\), the triple
\(\nu_3(F_k(\ell_0),\alpha_{k,r},F_k(\ell_0+2r))\) is minimal. The corresponding branch is
\[
\BB\big(\nu_1(F_k(\ell_0),\alpha_{k,r},F_k(\ell_0+2r))\big)
=
\BB\big(\alpha_{k,r},F_k(\ell_0+2r),F_k(\ell_0+4r)\big).
\]
\end{enumerate}
\end{theorem}

\begin{proof}
(1) Assume that  $\ell_0\in\{2,4,\dots,2r\}$, and fix $r$ odd.
For the triples $(F_k(\ell_0),\alpha_{k,r},F_k(\ell_0+2r))$ we consider
\[
\Phi(\ell_0):=\Phi(F_k(\ell_0),\alpha_{k,r},F_k(\ell_0+2r))=F_k(\ell_0+2r)-3\alpha_{k,r}F_k(\ell_0).
\]

We claim that $\Phi(\ell_0)$ is strictly decreasing as $\ell_0$ runs over the even
integers. Indeed, 
\begin{align*}
\mm\bigl(F_k(\ell_0),\alpha_{k,r},F_k(\ell_0+2r)\bigr)
&=F_k(\ell_0)^2+\alpha_{k,r}^2+F_k(\ell_0+2r)^2
-3\alpha_{k,r}F_k(\ell_0)F_k(\ell_0+2r)\\
&=F_k(\ell_0)^2+\alpha_{k,r}^2
+F_k(\ell_0+2r)\bigl(F_k(\ell_0+2r)-3\alpha_{k,r}F_k(\ell_0)\bigr)\\
&=F_k(\ell_0)^2+\alpha_{k,r}^2+F_k(\ell_0+2r)\,\Phi(\ell_0).
\end{align*}
By Proposition~\ref{thm:familia_par_impar}, the quantity
$\mm\bigl(\alpha_{k,r},F_k(\ell_0),F_k(\ell_0+2r)\bigr)$ is constant when $\ell_0$
is even. Since $\alpha_{k,r}$ is fixed and both $F_k(\ell_0)$ and $F_k(\ell_0+2r)$
strictly increase with $\ell_0$, the identity above forces $\Phi(\ell_0)$ to be
strictly decreasing along even $\ell_0$.

Finally, when $\ell_0=2r$, Lemma~\ref{lemma:anitha} yields
\[
\Phi(2r)=F_k(4r)-3\alpha_{k,r}F_k(2r)
=F_k(4r)-\frac{F_k(4r)}{F_k(2r)}\,F_k(2r)=0.
\]
Since $\Phi(\ell_0)$ is strictly decreasing on even $\ell_0$, we have
$\Phi(\ell_0)>0$ for all $\ell_0\in\{2,4,\dots,2r\}$, hence the triples
$(F_k(\ell_0),\alpha_{k,r},F_k(\ell_0+2r))$ are minimal for these values.

Moreover, Lemma~\ref{lemma:enganche} shows that
\[
\nu_1\bigl(F_k(\ell_0),\alpha_{k,r},F_k(\ell_0+2r)\bigr)
=(\alpha_{k,r},F_k(\ell_0+2r),F_k(\ell_0+4r)),
\]
and repeated application of $\nu_2$ to this triple yields the branch
\[
\BB\bigl(\nu_1(F_k(\ell_0),\alpha_{k,r},F_k(\ell_0+2r))\bigr)
=\BB\bigl(\alpha_{k,r},F_k(\ell_0+2r),F_k(\ell_0+4r)\bigr).
\]

(2) Assume now that $\ell_0\in\{1,3,\dots,2r-1\}$, and that $k\ge4$.
We prove that the triple
\[
\nu_3\bigl(F_k(\ell_0),\alpha_{k,r},F_k(\ell_0+2r)\bigr)
=\bigl(3\alpha_{k,r}F_k(\ell_0)-F_k(\ell_0+2r),\;F_k(\ell_0),\;\alpha_{k,r}\bigr)
\]
is minimal. Since minimality is equivalent to the inequality
\[
\alpha_{k,r}\ \ge\ 3F_k(\ell_0)\bigl(3\alpha_{k,r}F_k(\ell_0)-F_k(\ell_0+2r)\bigr),
\]
and using $\alpha_{k,r}=F_k(4r)/(3F_k(2r))$, it suffices to show that
\[
\frac{F_k(4r)}{3F_k(2r)}
-\frac{3F_k(4r)}{F_k(2r)}\,F_k(\ell_0)^2
+3F_k(\ell_0)F_k(\ell_0+2r)\ge0.
\]
Equivalently, we must prove that
\begin{equation}\label{eqn:Gl}
g(\ell_0):=
F_k(4r)-9F_k(4r)F_k(\ell_0)^2+9F_k(2r)F_k(\ell_0)F_k(\ell_0+2r)\ge 0.
\end{equation}

Using the sum identity \eqref{eq:sum} with \(a=2r\) and \(b=\ell_0\), we have
\[
F_k(\ell_0+2r)=F_k(2r+1)F_k(\ell_0)+F_k(2r)F_k(\ell_0-1).
\]
Substituting this into \eqref{eqn:Gl} and expanding gives
\[
\begin{aligned}
g(\ell_0)
=\;&9F_k(2r)F_k(\ell_0)^2F_k(2r+1)
+9F_k(2r)^2F_k(\ell_0)F_k(\ell_0-1)\\
&\qquad+F_k(4r)-9F_k(4r)F_k(\ell_0)^2.
\end{aligned}
\]
By \eqref{eqn:fibonacci1}, we have $F_k(4r)=F_k(2r)\bigl(F_k(2r+1)+F_k(2r-1)\bigr)$, so the
terms containing $F_k(2r)F_k(\ell_0)^2F_k(2r+1)$ cancel, and we obtain
\[
g(\ell_0)=F_k(2r)\Bigl[
9F_k(\ell_0)\bigl(F_k(2r)F_k(\ell_0-1)-F_k(\ell_0)F_k(2r-1)\bigr)
+F_k(2r+1)+F_k(2r-1)
\Bigr].
\]

\noindent Apply D’Ocagne’s identity \eqref{eq:DOcagne} with $a=\ell_0-1$ and $b=2r-1$
(note that $b\ge a$ since $\ell_0\le 2r-1$). Then
\[
(-1)^{\ell_0-1}F_k(2r-\ell_0)
=F_k(2r-1)F_k(\ell_0)-F_k(2r)F_k(\ell_0-1).
\]
Since $\ell_0$ is odd, $(-1)^{\ell_0-1}=1$, hence
\[
F_k(2r)F_k(\ell_0-1)-F_k(\ell_0)F_k(2r-1)=-F_k(2r-\ell_0).
\]
Therefore,
\[
g(\ell_0)=F_k(2r)\Bigl[
-9F_k(\ell_0)F_k(2r-\ell_0)+F_k(2r+1)+F_k(2r-1)
\Bigr].
\]

\noindent Finally, by the product bound \eqref{eq:bound_product}, we have
\[
F_k(\ell_0)F_k(2r-\ell_0)\le F_k(2r-1)
\]
and using the recurrence $F_k(2r+1)=kF_k(2r)+F_k(2r-1)$ we obtain
\[
F_k(2r+1)+F_k(2r-1)=kF_k(2r)+2F_k(2r-1).
\]
Hence
\[
\begin{aligned}
g(\ell_0)
&\ge F_k(2r)\Bigl[-9F_k(2r-1)+kF_k(2r)+2F_k(2r-1)\Bigr]\\
&=F_k(2r)\Bigl[kF_k(2r)-7F_k(2r-1)\Bigr]\\
&=F_k(2r)\Bigl[(k^2-7)F_k(2r-1)+kF_k(2r-2)\Bigr],
\end{aligned}
\]
where in the last step we used $F_k(2r)=kF_k(2r-1)+F_k(2r-2)$.
For $k\ge4$ and $r\ge1$, both terms in brackets are non-negative (and not both zero),
so $g(\ell_0)>0$.
\end{proof}

\section{Non-existence of other infinite paths}
\label{section:6}

In this section, we prove Theorem~\ref{thm:only_principal}.
Recall that an infinite path is an infinite sequence of 
Markoff $m$-triples $\{(a_n,b_n,c_n)\}_{n\ge 0}$ such that each triple is obtained
from the previous one by applying one of the Vieta transformations (see \eqref{def:vieta}) $\nu_1$ or $\nu_2$.
Equivalently, for every $n\ge 1,$ we have
\begin{equation}
\nu_3(a_{n+1},b_{n+1},c_{n+1})=(a_n,b_n,c_n).
\end{equation}
We say that an infinite path is $(2,k)$-Fibonacci if each of its triples contains
at least two $k$-Fibonacci components. Finally, recall that a \emph{principal}
$(2,k)$-Fibonacci branch is a branch of the form
\[
\BB\bigl(\frac{F_k(4r)}{3F_k(2r)},\,F_k(N-r),\,F_k(N+r)\bigr),
\]
where $\bigl(\frac{F_k(4r)}{3F_k(2r)},F_k(N-r),F_k(N+r)\bigr)$ is an ordered
Markoff $m$-triple with $m>0$.

\begin{lemma}
\label{lemma:no_a}
Let $\{(a_n,b_n,c_n)\}$ be an infinite $(2,k)$-Fibonacci path which is not contained
in any principal $(2,k)$-Fibonacci branch. Then it contains only finitely many triples
for which $a_n$ is not a $k$-Fibonacci number, and only finitely many triples having
three $k$-Fibonacci components.
\end{lemma}

\begin{proof}
If $a_n$ is not a $k$-Fibonacci number for some $n$, then by Theorem \ref{the:nonminimal_general} the triple must be of the form $(a_n,b_n,c_n)=\left ( \frac{F_k(4r)}{3F_k(2r)}, F_k(N-r),F_k(N+r)\right)$. 

From the discussion in Section \ref{section:5}, we have generically
$$(a_{n-1},b_{n-1},c_{n-1})=\nu_3(a_n,b_n,c_n) =\left ( \frac{F_k(4r)}{3F_k(2r)}, F_k(N-3r),F_k(N-r)\right)$$
except maybe for $n=2$, where the expressions for $a_1$ and $b_1$ might be switched in the minimal element $(a_1,b_1,c_1)$.

This implies that if $a_n$ is not $k$-Fibonacci for some $n$, then $a_j$ is also not $k$-Fibonacci for all $j$ with $1<j\le n$. As a consequence, if there is an infinite number of $n$ such that $a_n$ is not $k$-Fibonacci, then $a_n$ is not $k$-Fibonacci for every $n>1$, and the whole path
would eventually coincide with a principal $(2,k)$-Fibonacci branch, contrary to
our assumption.

Finally, by \cite[Theorem 1.1]{ACMRS25b}, if the path contains an infinite number of non-minimal triple
with three $k$-Fibonacci components, then it must be of the form $$(F_2(2)=2,F_2(2n),F_2(2n+2)),$$ which is also of the form given by Theorem \ref{the:nonminimal_general}, taking $r=1$ and $N=n+1$. Thus, repeating the previous argument, we can see that it only contains a finite number of these triples. As the path, by construction, can have at most one minimal triple, the result follows.
\end{proof}

Using the previous lemma, we can explicitly describe the structure of an infinite
$(2,k)$-Fibonacci path which is not contained in any principal branch.
Let $\{(a_n,b_n,c_n)\}$ be such a path. By Lemma~\ref{lemma:no_a}, after removing
finitely many initial triples if necessary, we may assume without loss of generality
that $a_n$ is always a $k$-Fibonacci number, that for every $n$ at least one of
$b_n$ or $c_n$ is not a $k$-Fibonacci number, and that all triples are non-minimal.

First observe that if $c_n$ is not a $k$-Fibonacci number for some $n$, then both
possible children of $(a_n,b_n,c_n)$ have $c_n$ as their second entry, so
$(a_{n+1},b_{n+1},c_{n+1})$ has a non-$k$-Fibonacci component in
$b_{n+1}=c_n$. Thus the path contains infinitely many triples for which $b_n$
is not a $k$-Fibonacci number.

Suppose now that $b_n$ is not a $k$-Fibonacci number for some $n$. Then, by
construction, $\nu_1(a_n,b_n,c_n)$ has $b_n$ as its smallest entry. Since
$a_{n+1}$ is a $k$-Fibonacci number, it follows that
\[
(a_{n+1},b_{n+1},c_{n+1})\neq \nu_1(a_n,b_n,c_n).
\]
Therefore,
\[
(a_{n+1},b_{n+1},c_{n+1})
=\nu_2(a_n,b_n,c_n)
=(a_n,c_n,3a_nc_n-b_n),
\]
and, since $a_n$ and $c_n$ are $k$-Fibonacci numbers, the entry $c_{n+1}$ is not
a $k$-Fibonacci number. In other words, the path must eventually alternate between
a triple $(a,b,c)$ where $b$ is not a $k$-Fibonacci number and a triple
$(a',b',c')=\nu_2(a,b,c)$ where $c'$ is not a $k$-Fibonacci number.

Therefore, after shifting the initial point of the path and reindexing if necessary, we may assume without loss of generality that such a path has the form
\begin{equation}
\label{eq:generic_branch}
\begin{aligned}
(a_{2n},b_{2n},c_{2n})&=(F_k(u_n),x_n,F_k(v_n)),\\
(a_{2n+1},b_{2n+1},c_{2n+1})&=(F_k(u_n),F_k(v_n),x_{n+1}),
\end{aligned}
\end{equation}
where
\[
x_{n+1}=3F_k(u_n)F_k(v_n)-x_n.
\]
Moreover, the next even-indexed triple is given by either
\[
(a_{2n+2},b_{2n+2},c_{2n+2})
=\nu_1(a_{2n+1},b_{2n+1},c_{2n+1})
=(F_k(u_{n+1}),x_{n+1},F_k(v_{n+1})),
\]
with \(u_{n+1}=v_n\), or
\[
(a_{2n+2},b_{2n+2},c_{2n+2})
=\nu_2(a_{2n+1},b_{2n+1},c_{2n+1})
=(F_k(u_{n+1}),x_{n+1},F_k(v_{n+1})),
\]
with \(u_{n+1}=u_n\).

Let us now find some lower and upper bounds for the non-$k$-Fibonacci terms $x_n$ in one such sequence.

\begin{lemma}
\label{lemma:b_bound}
Let $\{(a_n,b_n,c_n)\}_{n\ge0}$ be a path of the form \eqref{eq:generic_branch}
which is not contained in any principal $(2,k)$-Fibonacci branch. Then $\{u_n\}_{n\ge0}$ is non-decreasing, whereas
$\{x_n\}_{n\ge0}$ and $\{v_n\}_{n\ge0}$ are increasing. Moreover, for all sufficiently
large $n$, the following hold:
\begin{enumerate}
\item\label{lemma:b_bound_k12a}
If $k=1$ or $k=2$, then $x_n\in \bigl(F_k(v_n-u_n-1),\,F_k(v_n-u_n)\bigr).$

\item\label{lemma:b_bound_k12b}
If $k=1$ or $k=2$, then
$
x_{n+1}\in \bigl(F_k(u_n+v_n),\,F_k(u_n+v_n+1)\bigr).
$

\item\label{lemma:b_bound_kg3a}
If $k\ge 3$, then
$
x_n\in \bigl(F_k(v_n-u_n),\,F_k(v_n-u_n+1)\bigr).
$

\item\label{lemma:b_bound_kg3b}
If $k\ge 3$, then
$
x_{n+1}\in \bigl(F_k(u_n+v_n-1),\,F_k(u_n+v_n)\bigr).
$
\end{enumerate}
In particular, all the intervals above are open, since 
$x_n$  and $x_{n+1}$ are not $k$-Fibonacci numbers.
\end{lemma}

\begin{proof}
$a_n$ is not decreasing and $b_n$ and $c_n$ are increasing due to the construction using  repeated application of the Vieta transformations $\nu_1$ and $\nu_2$, and, thus, so are the indices $u_n$, $x_n$ and $v_n$.

Suppose first that $k=1$ or $k=2$. By Lemma \ref{lemma:F(c)<=3F(n)F(m)}, we know that for sufficiently large $n$ it is impossible that $x_n< F_k(v_n-u_n-1)$ since this would imply that $c_{2n}=F_k(v_n)>3F_k(u_n)F_k(v_n-u_n-1)>3a_{2n}b_{2n}$, contradicting that the triple is non-minimal. Thus, we have that $x_n > F_k(v_n-u_n-1)$. 

On the other hand, by \cite[Proposition 3.1]{ACMRS25a} in the case $k=1$ or by \cite[Lemma 3.1.(1) and Lemma 3.2]{ACMRS25b} in the case $k=2$ we have that $\mm(F_k(u_n),F_k(v_n-u_n),F_k(v_n))\le 0$. As $\mm(x,y,z)$ is monotonically decreasing in $y$ and $\mm(F_k(u_n),x_n,F_k(v_n))=m>0$, this implies that $x_n<F_k(v_n-u_n)$.

We can now apply a similar argument for $(a_{2n+1},b_{2n+1},c_{2n+1})=(F_k(u_n),F_k(v_n),x_{n+1})$. By Lemma \ref{lemma:F(c)<=3F(n)F(m)}, it is impossible that $x_{n+1}>F_k(u_n+v_n+1)$ since this would imply that $x_{n+1}>3F_k(u_n)F_k(v_n)$ and the triple would be non-minimal. Consequently, we have that $x_{n+1}< F_k(u_n+v_n+1)$. Similarly, by \cite[Proposition 3.1]{ACMRS25a} and \cite[Lemma 3.1.(1) and Lemma 3.2]{ACMRS25b}, we have that $\mm(F_k(u_n),F_k(v_n),F_k(u_n+v_n))<0$ for sufficiently large $n$ and $\mm(F_k(u_n),F_k(v_n),x_{n+1})=m>0$, and since $\mm(x,y,z)$ is monotonically decreasing in $z$, we have $x_{n+1}>F_k(u_n+v_n)$.

Analogously, for $k\ge 3$, non-minimality of the triple $(F_k(u_n),x_n,F_k(v_n))$ and Lemma \ref{lemma:F(c)<=3F(n)F(m)}(iii) imply that $x_n>F_k(v_n-u_n)$ and $x_{n+1}<F_k(v_n+u_n+1)$ and \cite[Lemma 3.1.(2)]{ACMRS25b} implies that $\mm(F_k(u_n),F_k(v_n-u_n+1),F_k(v_n))\le 0$ and $\mm(F_k(u_n),F_k(v_n),F_k(u_n+v_n))<0$, so $x_n<F_k(v_n-u_n+1)$ and $x_{n+1}>F_k(u_n+v_n)$.
\end{proof}

\begin{corol}
\label{cor:c-a_limit}
For a path of the form \eqref{eq:generic_branch} which is not contained in any principal $(2,k)$-Fibonacci branch,  we have $\lim_{n\to\infty} (v_n-u_n) = \infty$.
\end{corol}

\begin{proof}
Assume first that $k=1$ or $k=2$. By Lemma \ref{lemma:b_bound} for all $n$ sufficiently large we have that
$$x_{n+1}\in (F_k(u_n+v_n),F_k(u_n+v_n+1)),$$
and applying \ref{lemma:b_bound}(\ref{lemma:b_bound_k12a}) for the index $n+1,$ we also have
$$x_{n+1}\in (F_k(v_{n+1}-u_{n+1}-1),F_k(v_{n+1}-u_{n+1})),$$
and, since both are intervals of consecutive $k-$Fibonacci numbers, we must have that
\begin{equation}
    \label{eq:cn-an_recurrence_k12}
    v_{n+1}-u_{n+1}-1=u_n+v_n,
\end{equation}
but then
$$v_{n+1}-u_{n+1}=u_n+v_n+1>v_n-u_n$$
for each sufficiently large $n$. Thus, $\lim_{n\to\infty} (v_n-u_n)=\infty$. Analogously, if $k\ge 3$, by Lemma \ref{lemma:b_bound}(\ref{lemma:b_bound_kg3b}),
$$x_{n+1}\in (F_k(u_n+v_n-1),F_k(u_n+v_n)),$$
and by Lemma \ref{lemma:b_bound}(\ref{lemma:b_bound_kg3a}) applied to the index $n+1$,
$$x_{n+1}\in (F_k(v_{n+1}-u_{n+1}),F_k(v_{n+1}-u_{n+1}+1)),$$
and, by the same reasoning,
\begin{equation}
    \label{eq:cn-an_recurrence_kge3}
    v_{n+1}-u_{n+1}=u_n+v_n-1>v_n-u_n.
\end{equation}
Hence $\lim_{n\to\infty} (v_n-u_n) = \infty$.

\end{proof}

\begin{lemma}
\label{lemma:b_bound_strict}
Recall that $D_k = \sqrt{k^2+4}$. Assume that $k=1$ or $k=2$. Then for any path of the form \eqref{eq:generic_branch} as above, there exist $\varepsilon>0$ and $N>0$ such that one of the following holds:
\begin{itemize}
\item[i)] \label{lemma:b_bound_strict_i} $x_n\in \left( \frac{1}{D_k}\phi_k^{v_n-u_n-\frac{1}{2}+\varepsilon}, F_k(v_n-u_n)\right)\,, \quad \forall n\ge N$.
\item[ii)] \label{lemma:b_bound_strict_ii} $x_n\in \left(F_k(v_n-u_n-1), \frac{1}{D_k}\phi_k^{v_n-u_n-\frac{1}{2}-\varepsilon}\right)\,, \quad \forall n\ge N.$
\end{itemize}
Analogously, if $k\ge 3$ there exist $\varepsilon>0$ and $N>0$ such that one of the following holds:
\begin{itemize}
\item[iii)] \label{lemma:b_bound_strict_iii} $x_n\in \left( \frac{1}{D_k}\phi_k^{v_n-u_n+\frac{1}{2}+\varepsilon}, F_k(v_n-u_n+1)\right)\,, \quad \forall n\ge N$.
\item[iv)] \label{lemma:b_bound_strict_iv} $x_n\in \left(F_k(v_n-u_n), \frac{1}{D_k}\phi_k^{v_n-u_n+\frac{1}{2}-\varepsilon}\right)\,, \quad \forall n\ge N$.
\end{itemize}
\end{lemma}

\begin{proof}
Let us denote $A=\lim_{n\to\infty} u_n$. Since $u_n$ is non-decreasing by Lemma \ref{lemma:b_bound}, then $u_n$ is either eventually constant and equal to some integer $A>0$ or it diverges and $A=\infty$. 
We will begin by showing that the limit
$$L_{k,A}(\Delta):=\lim_{n\to\infty} \frac{1}{F_k(v_n)^2} \,\,\mm\!\left(F_k(u_n),\frac{1}{D_k} \phi_k^{v_n-u_n+\Delta},F_k(v_n)\right)$$
is always finite and it only depends on $k$, $A$ and $\Delta$. Expanding the expression of $\mm(a,b,c)$ we have
\begin{multline*}
\frac{1}{F_k(v_n)^2}\,\, \mm\!\left(F_k(u_n),\frac{1}{D_k} \phi_k^{v_n-u_n+\Delta},F_k(v_n)\right) = \frac{F_k(u_n)^2}{F_k(v_n)^2} + 1 + \frac{\phi_k^{2v_n-2u_n+2\Delta}}{D_k^2F_k(v_n)^2} -3 \frac{F_k(u_n) \phi_k^{v_n-u_n+\Delta}}{D_kF_k(v_n)}.
\end{multline*}
Now, since $v_n-u_n\to \infty$ by Corollary \ref{cor:c-a_limit}, together with \eqref{eqn:usefulasympt}, we have that
$$\lim_{n\to \infty} \frac{F_k(u_n)^2}{F_k(v_n)^2} = \lim_{n\to \infty} \phi_k^{-2(v_n-u_n)}=0.$$
$$\lim_{n\to\infty} \frac{\phi_k^{2v_n-2u_n+2\Delta}}{D_k^2F_k(v_n)^2} = \lim_{n\to\infty}\phi_k^{2\Delta-2u_n} = \begin{cases}
    \phi_k^{2\Delta-2A} & A<\infty,\\
    0& A=\infty.
\end{cases}$$
$$\lim_{n\to\infty} \frac{F_k(u_n) \phi_k^{v_n-u_n+\Delta}}{D_kF_k(v_n)} = \begin{cases}
    \frac{F_k(A)}{\phi_k^A} \phi_k^{\Delta} & A<\infty,\\
    \frac{1}{D_k} \phi_k^\Delta & A=\infty.
\end{cases}$$
Putting the previous computations together,
$$L_{k,A}(\Delta)=\begin{cases}
    1+\phi_k^{2\Delta-2A}-\frac{3F_k(A)}{\phi_k^A} \phi_k^{\Delta} & A<\infty,\\
    1-\frac{3}{D_k} \phi_k^\Delta & A=\infty,
\end{cases}$$
which is indeed always a finite real number. Now, we observe that if $\Delta= \frac{1}{2}$ or $\Delta=-\frac{1}{2}$ then the limit $L_{k,A}(\Delta)$ can never be zero. To show this, assume the contrary. Suppose that $L_{k,A}(1/2)=0$. This is impossible if $A=\infty$, since $L_{k,\infty}(1/2)=0$ would imply
$\phi_k^{1/2}=D_k/3$, and hence $\phi_k=(k^2+4)/9$, contradicting the irrationality of $\phi_k$. 

If $A<\infty$, then
\[
1+\phi_k^{1-2A}-3\frac{F_k(A)}{\phi_k^A}\phi_k^{1/2}=0,
\]
and therefore
\[
\sqrt{\phi_k}
=\frac{(1+\phi_k^{1-2A})\phi_k^A}{3F_k(A)}
=\frac{\phi_k^A+\phi_k^{1-A}}{3F_k(A)}.
\]
Since $A$ is an integer, the right-hand side belongs to $\mathbb{Q}(\phi_k)$. Thus
$\sqrt{\phi_k}\in\mathbb{Q}(\phi_k)$, so $\phi_k$ is a square in the quadratic field
$\mathbb{Q}(\phi_k)$. However, the field norm from $\mathbb{Q}(\phi_k)$ to $\mathbb{Q}$ satisfies
\[
N_{\mathbb{Q}(\phi_k)/\mathbb{Q}}(\phi_k)=\phi_k\overline{\phi}_k=-1,
\]
whereas the norm of a square is always a square in $\mathbb{Q}$. This is impossible.
Therefore $L_{k,A}(1/2)\neq 0$.

Assume first that $L_{k,A}(1/2)>0$. By continuity of $L_{k,A}(\Delta)$, there exists
$\varepsilon>0$ such that $L_{k,A}(1/2+\varepsilon)>0$. Since $F_k(v_n)\to\infty$, it follows that
\[
\mm\!\left(F_k(u_n),\frac{1}{D_k}\phi_k^{v_n-u_n+\frac12+\varepsilon},F_k(v_n)\right)\to\infty.
\]
Hence there exists $N>0,$ such that
\[
\mm\!\left(F_k(u_n),\frac{1}{D_k}\phi_k^{v_n-u_n+\frac12+\varepsilon},F_k(v_n)\right)>0\,,
\qquad\forall n\ge N.
\]
Since $\mm(x,y,z)$ is decreasing in $y$, we deduce that
\[
x_n>\frac{1}{D_k}\phi_k^{v_n-u_n+\frac12+\varepsilon}\,,
\qquad\forall n\ge N.
\]

\noindent Similarly, if $L_{k,A}(1/2)<0,$ then

$$
x_n<\frac{1}{D_k}\phi_k^{v_n-u_n+\frac12-\varepsilon}\,,
\qquad\forall n\ge N.
$$


If $k\ge 3$, by Lemma~\ref{lemma:b_bound}(\ref{lemma:b_bound_kg3a}) we know that
\[
x_n\in \bigl(F_k(v_n-u_n),\,F_k(v_n-u_n+1)\bigr).
\]
Combining this with the previous inequalities, alternatives \textup{(iii)} and \textup{(iv)} follow.

Then, alternatives (i) and (ii) follow analogously from Lemma \ref{lemma:b_bound}(\ref{lemma:b_bound_k12a}), repeating the previous argument with $\Delta=-\frac{1}{2}$. 
\end{proof}

\begin{theorem}[Theorem \ref{thm:only_principal} of the Introduction]
Let $m>0$ and $k\ge1$. Every infinite path of Markoff $m$-triples
with at least two $k$-Fibonacci components is contained in a principal $(2,k)$-Fibonacci branch.
\end{theorem}

\begin{proof}
Suppose, on the contrary, that there exists an infinite $(2,k)$-Fibonacci path \\$\{(a_n,b_n,c_n)\}_{n\ge0}$
which is not contained in any principal $(2,k)$-Fibonacci branch.
By the previous discussion, after removing finitely many initial triples if necessary,
we may assume that it has the form \eqref{eq:generic_branch} for some sequences
$u_n$, $x_n$ and $v_n$. Suppose first that $k=1$ or $k=2$. By Lemma \ref{lemma:b_bound_strict}, there exists $\varepsilon>0$ and $N>0$ such that either
$$x_n>\frac{1}{D_k} \phi_k^{v_n-u_n-\frac{1}{2}+\varepsilon}\,, \quad \forall n\ge N,$$
or
\begin{equation}
    \label{eqn:alternative2}
x_n<\frac{1}{D_k} \phi_k^{v_n-u_n-\frac{1}{2}-\varepsilon}\,, \quad \forall n\ge N.
\end{equation}
First, assume the former. Then, by \eqref{eq:cn-an_recurrence_k12}, for all $n\ge N$ we also have that
$$x_{n+1}>\frac{1}{D_k} \phi_k^{v_{n+1}-u_{n+1}-\frac{1}{2}+\varepsilon}=\frac{1}{D_k} \phi_k^{v_n+u_n+\frac{1}{2}+\varepsilon}\,, \quad \forall n\ge N.$$
Thus,
$$x_n x_{n+1}>\frac{1}{D_k^2}\phi_k^{2v_n+2\varepsilon}.$$
Since $(a_{2n}, b_{2n}, c_{2n}) = (F_k(u_n), x_n, F_k(v_n))$ and $(a_{2n+1}, b_{2n+1}, c_{2n+1}) = (F_k(u_n), F_k(v_n), x_{n+1})$, it follows that $x_n$ and $x_{n+1}$ are the two solutions to the quadratic equation 
$$x^2-3F_k(u_n)F_k(v_n)x + F_k(u_n)^2+F_k(v_n)^2-m=0,$$ so $x_n x_{n+1}=F_k(u_n)^2+F_k(v_n)^2-m$. This means that
$$1=\frac{F_k(u_n)^2+F_k(v_n)^2-m}{x_n x_{n+1}} <\frac{D_k^2(F_k(u_n)^2+F_k(v_n)^2-m)}{\phi_k^{2v_n}} \phi_k^{-2\varepsilon}\,, \quad \forall n\ge N,$$
but, as $\lim_{n\to\infty} (v_n-u_n)=\infty$ by Corollary \ref{cor:c-a_limit}, together with \eqref{eqn:usefulasympt}, 
the limit of the right hand side is $\phi_k^{-2\varepsilon}$. Since $\varepsilon>0$, this limit is strictly less than $1$, which is impossible. Analogously, if \eqref{eqn:alternative2} holds, then
$$x_{n+1}<\frac{1}{D_k} \phi_k^{v_n+u_n+\frac{1}{2}-\varepsilon},$$
and we would have
$$1=\frac{F_k(u_n)^2+F_k(v_n)^2-m}{x_n x_{n+1}} >\frac{D_k^2(F_k(u_n)^2+F_k(v_n)^2-m)}{\phi_k^{2v_n}} \phi_k^{2\varepsilon}\,, \quad \forall n\ge N,$$
which is impossible, since the limit of the right-hand side is $\phi_k^{2\varepsilon}>1$,
again by Corollary \ref{cor:c-a_limit} and \eqref{eqn:usefulasympt}. The result is completely analogous if $k\ge 3$. In this case, either
$$x_n>\frac{1}{D_k} \phi_k^{v_n-u_n+\frac{1}{2}+\varepsilon}\,, \quad \forall n\ge N,$$
or
$$x_n<\frac{1}{D_k} \phi_k^{v_n-u_n+\frac{1}{2}-\varepsilon}\,, \quad \forall n\ge N,$$
and using \eqref{eq:cn-an_recurrence_kge3} we have respectively
$$x_{n+1}>\frac{1}{D_k} \phi_k^{v_n+u_n-\frac{1}{2}+\varepsilon}\,, \quad \forall n\ge N,$$
or
$$x_{n+1}<\frac{1}{D_k} \phi_k^{v_n+u_n-\frac{1}{2}-\varepsilon}\,, \quad \forall n\ge N.$$
Thus we have $x_n x_{n+1}>\frac{1}{D_k^2} \phi_k^{2v_n+2\varepsilon}$ or $x_n x_{n+1}<\frac{1}{D_k^2}\phi_k^{2v_n-2\varepsilon}$ respectively, and we can apply the same argument as before.
\end{proof}

\section{Examples}\label{section:examples}
We conclude by presenting several examples of principal $(2,k)$-Fibonacci branches in different Markoff trees, illustrating the results proved in the previous sections.

\begin{example}[The case $r=1$]
Consider the family of triples
$
(\alpha_{k,1},F_k(\ell),F_k(\ell+2))$ with $\ell\ge 2,
$ and $3\nmid k$. In this case there are two branches, determined by the initial triples
$$(F_k(2),\alpha_{k,1},F_k(4))=\left(k,\frac{k^2+2}{3},k^3+2k\right) \quad (F_k(1),\alpha_{k,1},F_k(3))=\left(1,\frac{k^2+2}{3},k^2+1\right).$$
When $m>0$, these are the corresponding minimal triples.

The even-parity branch, corresponding to even values of $\ell\ge 2$, satisfies
\[
\mm(\alpha_{k,1},F_k(\ell),F_k(\ell+2))
=
\mm\!\left(k,\frac{k^2+2}{3},k^3+2k\right)
=
\frac{1}{9}(k^4+13k^2+4).
\]

The odd-parity branch, corresponding to odd values of $\ell\ge 1$, satisfies
\[
\mm(\alpha_{k,1},F_k(\ell),F_k(\ell+2))
=
\mm\!\left(1,\frac{k^2+2}{3},k^2+1\right)
=
\frac{1}{9}(k^4-5k^2+4).
\]
\end{example}

\medskip

The next two examples show that our results are consistent with previously known cases in the literature.

\begin{example}[The Fibonacci case: $r=1$, $k=1$]
In this case, the initial triples are $
(1,1,3)$ and 
$(1,1,2).$
The first one is minimal, whereas the second is not, since in this case $m=0$. The even-parity branch satisfies
\[
\mm(1,F_1(n),F_1(n+2))=2,
\]
whereas the odd-parity branch satisfies
\[
\mm(1,F_1(n),F_1(n+2))=0.
\]
These two principal $(2,1)$-branches and trees agree with the results obtained in \cite{ACMRS25a} and \cite{LS}, respectively.
\end{example}

\medskip

\begin{example}[The Pell case: $r=1$, $k=2$] In this case, the two  initial triples are $(2,2,12)$ and $(1,2,5).$
Only the first one is minimal. In this case, the even-parity branch satisfies
\[
\mm(2,F_2(n),F_2(n+2))=8,
\]
whereas the odd-parity branch satisfies
\[
\mm(2,F_2(n),F_2(n+2))=0.
\]
Once again, this is consistent with \cite{ACMRS25b} and \cite{KST}, respectively.
\end{example}

\medskip

We now present an example that does not seem to have appeared previously in the literature.

\begin{example}[Case $r=1$ and $k=4$ ] In this case $\alpha_{4,1}=6$. There are two $(2,4)$-Fibonacci branches corresponding to the initial triples $\{(4,6,72), (1,6,17)\}$. 
Once again, the first is a minimal triple and the second is not, but $\nu_3(1,6,17)=(1,1,6)$ is minimal.

In this case, the even-parity branch satisfies that
$\mm(6,F_{4}(n),F_{4}(n+2))=52$ while the odd-parity branch has  $\mm(6,F_{4}(n),F_{4}(n+2))=20$. 
\end{example}
\medskip 

\begin{figure}[H]
\centering
\begin{tikzpicture}[grow'=right, edge from parent/.style={draw, -latex}]
\tikzset{
  level 1/.style={level distance=3.5cm, sibling distance=0.35cm},
  level 2/.style={level distance=4.5cm, sibling distance=0.35cm},
  level 3/.style={level distance=5.5cm, sibling distance=0.35cm}
}
\Tree 
[. (4,6,72)
    [.\textbf{(6,72,1292)}
        [.\textbf{(6,1292,23184)}
            [.\textbf{(6,23184,416020)} ]
            [.(1292,23184,89861178) ]
        ]
        [.(72,1292,279066)
            [.(72,279066,60276964) ]
            [.(1292,279066,1081659744) ]
        ]
    ]
]
\end{tikzpicture}
\caption{Beginning of the Markoff $52$-tree, with minimal triple $(4,6,72)$. The bold path shows the principal $(2,4)$-Fibonacci branch formed by the triples $\bigl(6,F_4(n),F_4(n+2)\bigr)$, where $n$ is even.}
\label{52-markovtree}
\end{figure}

\begin{figure}[H]
\centering
\begin{tikzpicture}[grow'=right, edge from parent/.style={draw, -latex}]
\tikzset{
  level 1/.style={level distance=3.5cm, sibling distance=0.35cm},
  level 2/.style={level distance=4.5cm, sibling distance=0.35cm},
  level 3/.style={level distance=5.5cm, sibling distance=0.35cm}
}
\Tree 
[. (1,1,6)
    [.(1,6,17)
        [.\textbf{(6,17,305)}
            [.\textbf{(6,305,5473)} ]
            [.(17,305,15549) ]
        ]
        [.(1,17,45)
            [.(1,45,118) ]
            [.(17,45,2294) ]
        ]
    ]
]
\end{tikzpicture}
\caption{Beginning of the Markoff $20$-tree with minimal triple $(1,1,6)$. The bold path shows the principal $(2,4)$-Fibonacci branch formed by the triples $\bigl(6,F_4(n),F_4(n+2)\bigr)$, where $n$ is odd.}
\label{20-markovtree-odd}
\end{figure}

Finally, we compute the branches in an example with $r>1.$

\begin{example}[Case $r=3$ and $k=1$] In this case $\alpha_{1,3}=6$. We consider even cases ($\ell=2,4,6$).  They all have Markoff constant $m=100$. In the odd cases, by Proposition \ref{prop2.5}, we obtain $m <0$.

\begin{figure}[H]
\centering
\begin{tikzpicture}[grow'=right, edge from parent/.style={draw, -latex}]
\tikzset{
  level 1/.style={level distance=3.5cm, sibling distance=0.35cm},
  level 2/.style={level distance=4.5cm, sibling distance=0.35cm},
  level 3/.style={level distance=5.5cm, sibling distance=0.35cm}
}
\Tree 
[. (1,6,21)
    [.\textbf{(6,21,377)}
        [.\textbf{(6,377,6765)}
            [.\textbf{(6,6765,121393)} ]
            [.(377,6765,7651209) ]
        ]
        [.(21,377,23745)
            [.(21,23745,1495558) ]
            [.(377,23745,26855574) ]
        ]
    ]
]
\end{tikzpicture}
\caption{Beginning of the Markoff $100$-tree with minimal triple $(1,6,21)$. The bold path shows the principal $(2,1)$-Fibonacci branch formed by the triples $\bigl(6,F_1(n),F_1(n+6)\bigr)$, where $n$ is even.}
\label{100-markovtree-ell-2-fix}
\end{figure}

\begin{figure}[H]
\centering
\begin{tikzpicture}[grow'=right, edge from parent/.style={draw, -latex}]
\tikzset{
  level 1/.style={level distance=3.5cm, sibling distance=0.35cm},
  level 2/.style={level distance=4.5cm, sibling distance=0.35cm},
  level 3/.style={level distance=5.5cm, sibling distance=0.35cm}
}
\Tree 
[. (3,6,55)
    [.\textbf{(6,55,987)}
        [.\textbf{(6,987,17711)}
            [.\textbf{(6,17711,317811)} ]
            [.(987,17711,52442265) ]
        ]
        [.(55,987,162849)
            [.(55,162849,26869098) ]
            [.(987,162849,482195834) ]
        ]
    ]
]
\end{tikzpicture}
\caption{Beginning of the Markoff $100$-tree with minimal triple $(3,6,55)$. The bold path shows the principal $(2,1)$-Fibonacci branch formed by the triples $\bigl(6,F_1(n),F_1(n+6)\bigr)$, where $n$ is even.}
\label{100-markovtree-ell-4-fix}
\end{figure}

\begin{figure}[H]
\centering
\begin{tikzpicture}[grow'=right, edge from parent/.style={draw, -latex}]
\tikzset{
  level 1/.style={level distance=3.5cm, sibling distance=0.35cm},
  level 2/.style={level distance=4.5cm, sibling distance=0.35cm},
  level 3/.style={level distance=5.5cm, sibling distance=0.35cm}
}
\Tree 
[. (6,8,144)
    [.\textbf{(6,144,2584)}
        [.\textbf{(6,2584,46368)}
            [.\textbf{(6,46368,832040)} ]
            [.(2584,46368,359444730) ]
        ]
        [.(144,2584,1116282)
            [.(144,1116282,482231240) ]
            [.(2584,1116282,8653418058) ]
        ]
    ]
]
\end{tikzpicture}
\caption{Beginning of the Markoff $100$-tree with minimal triple $(6,8,144)$. The bold path shows the principal $(2,1)$-Fibonacci branch formed by the triples $\bigl(6,F_1(n),F_1(n+6)\bigr)$, where $n$ is even.}
\label{100-markovtree-ell-6-fix}
\end{figure}
\end{example}


\end{document}